\documentclass[12pt,a4paper]{article}
\usepackage{amsmath}
\usepackage{amssymb}
\usepackage{amsthm}
\usepackage{amsfonts}
\usepackage{latexsym}

\theoremstyle{plain}
\newtheorem{thm}{Theorem}[section]
\newtheorem{cor}{Corollary}[section]
\newtheorem{lem}{Lemma}[section]

\theoremstyle{remark}

\theoremstyle{definition}
\newtheorem{defn}{Definition}[section]

\newtheorem{exmp}{Example}[section]

\newtheorem{rem}{Remark}[section]
\newtheorem{asmpt}{Assumption}[section]

\title{Szeg\"{o} kernel equivariant asymptotics under Hamiltonian Lie group actions}
\author{Roberto Paoletti\footnote{\noindent{\bf Address:}
Dipartimento di Matematica e Applicazioni, Universit\`a degli Studi
di Milano-Bicocca, Via R. Cozzi 55, 20125 Milano,
Italy; {\bf e-mail}: roberto.paoletti@unimib.it }}
\date{}

\begin{document}
\maketitle

\begin{abstract}
Suppose that a compact and connected Lie group $G$
acts on a complex Hodge 
manifold $M$ in a holomorphic and Hamiltonian manner, and that the action linearizes
to a positive holomorphic line bundle $A$ on $M$. Then there is an induced unitary representation
on the associated Hardy space and, if the moment map of the action is nowhere vanishing,
the corresponding isotypical components are all finite dimensional. We study the asymptotic concentration
behavior of the corresponding equivariant Szeg\"{o} kernels near certain loci
defined by the moment map.
\end{abstract}

\section{Introduction}

Let $M$ be a connected complex $d$-dimensional projective manifold,
and $A$ an holomorphic ample line bundle on it. There exists an Hermitian metric
$h$ such that the unique covariant $\nabla$ derivative on $A$ that is compatible with both
the complex structure and the metric has curvature $\Theta=-2\,\imath\,\omega$, where
$\omega$ is a K\"{a}hler form on $M$. Thus the triple $(M,J,\omega)$ is a 
K\"{a}hler manifold, with associated Riemannian metric $\rho^M$ and volume form
$\mathrm{d}V_M:=\omega^{\wedge d}/d!$.

We shall denote by $A^\vee$ the dual line bundle of $A$, and by $X\subset A^\vee$
the unit circle bundle; thus $X=\partial D$, where $D\subset A^\vee$ is the unit disc bundle,
a strictly pseudoconvex domain. Then $\nabla$ determines a connection
$1$-form $\alpha$ on $X$. If $\pi:X\rightarrow M$ is the projection, then 
$\mathrm{d}V_X:=\frac{1}{•2\,\pi}\,\alpha\wedge \pi^*(\mathrm{d}V_M)$
is a volume form. 
Furthermore, there is on $X$ a natural choice of an $S^1$-invariant 
Riemannian metric $\rho^X$, determined by
the conditions that $\pi$ be a Riemannian submersion with $\ker (\alpha)$ as horizontal tangent
bundle, and that the fibers of $\pi$ have unit length.
Hence $|\mathrm{d}V_X|$ is the Riemannian density of $\rho^X$.
We shall denote by $\mathrm{dist}_X$ the Riemannian distance function of $\rho^X$.

If $H(X)\subset L^2(X)$ is the Hardy space, the orthogonal projector 
$\Pi:L^2(X)\rightarrow H(X)$ is known as the \textit{Szeg\"{o} projector}  of $X$, and its distributional kernel
$\Pi\in \mathcal{D}'(X\times X)$ as the \textit{Szeg\"{o} kernel} (\cite{bs}, \cite{z}).

Furthermore, let $G$ be a connected compact  
Lie group, with Lie algebra $\mathfrak{g}$ and
coalgebra $\mathfrak{g}^\vee$; 
we shall denote by $d_G$ and $r_G$, respectively, the dimension and the rank of
$G$. 
Let us assume that $\mu:G\times M\rightarrow M$ is a 
Hamiltonian and holomorphic action on $(M,J,2\,\omega)$,
with moment map $\Phi:M\rightarrow \mathfrak{g}^\vee$.

Then to every $\boldsymbol{\xi}\in \mathfrak{g}$
there is associated an Hamiltonian vector field $\boldsymbol{\xi}_M\in \mathfrak{X}_{ham}(M)$, which
canonically lifts to a contact and CR vector field
$\boldsymbol{\xi}_X\in \mathfrak{X}_{cont}(X)$, according to the law \cite{ko}
\begin{equation}
\label{eqn:contact}
\boldsymbol{\xi}_X:=\boldsymbol{\xi}_M^\sharp-\langle\Phi,\boldsymbol{\xi}\rangle\,\partial_\theta;
\end{equation}
here notation is as follows:
\begin{enumerate}
\item for any vector field $V$ on $M$, $V^\sharp$ denotes its its horizontal lift to $X$ 
with respect to $\alpha$ (and similarly for tangent vectors);
\item $\partial_\theta$ is the generator of the standard circle action on $X$ (fiber rotation). 
\end{enumerate}
In other words, the infinitesimal action of $\mathfrak{g}$ on $M$ lifts
to an infinitesimal contact action on $X$.

We shall make the stronger hypothesis that $\mu$ itself lifts to an action  
$\tilde{\mu}:G\times X\rightarrow X$, of which the correspondence $\boldsymbol{\xi}\mapsto \boldsymbol{\xi}_X$
is the differential.
Then $\tilde{\mu}$ is a contact and CR action, and naturally determines a unitary representation
of $G$ on $H(X)$. According to the Theorem of Peter and Weyl, 
there is an equivariant unitary Hilbert direct sum 
decomposition of $H(X)$ into isotypical components corresponding to the irreducible representations
(in the following, irreps) of $G$
\cite{st}.

For a given choice of a maximal torus 
$T\leqslant G$
and of a set $R^+$ of positive roots of $\mathfrak{g}$, the irreps of $G$ are
determined by the their maximal weights, which range in the set of dominant weights.
This sets up a bijective correspondence between the family $\hat{G}$ of irreps of
$G$ and a subset
$\mathcal{D}^G$ of the collection $\mathcal{D}$ of all dominant weights;
we have $\mathcal{D}=\mathcal{D}^G$ if $G$ is simply connected. We shall 
label the irrep with maximal weight $\boldsymbol{\lambda}
\in \mathcal{D}^G$ by the regular\footnote{i.e., belonging to the open positive Weyl chamber} 
half-weight $\boldsymbol{\nu}=\boldsymbol{\lambda}+\boldsymbol{\delta}$,
where $\boldsymbol{\delta}:=2^{-1}\,\sum_{\beta\in R^+}\beta$. 
We shall denote by $V_{\boldsymbol{\nu}}$ the irrep corresponding to 
$\boldsymbol{\nu}$, and by $\chi_{\boldsymbol{\nu}}:G\rightarrow \mathbb{C}$ the corresponding character. 
This labeling is consistent with the philosophy of the Kirillov character formula \cite{ki}, that
we shall recall in the course of the paper. 
Furthermore, let us set
$d_{\boldsymbol{\nu}}:=\dim (V_{\boldsymbol{\nu}})$; if 
$\varphi$ denotes any Euclidean scalar product on $\mathfrak{g}^\vee$
which is invariant uder the coadjoint action then, by the Weyl dimension formula
(see, e.g., \S 1 of \cite{su}, \S 2.5 of \cite{v1}),
\begin{equation}
\label{eqn:weyldf}
d_{\boldsymbol{\nu}}=\prod_{\boldsymbol{\beta}\in R^+}\frac{\varphi (\boldsymbol{\nu},\boldsymbol{\beta})}{\varphi (\boldsymbol{\delta},\boldsymbol{\beta})}.
\end{equation}
In particular, $d_{\boldsymbol{k\,\nu}}=k^{\frac{1}{2}\,( d_G-r_G)}\,d_{\boldsymbol{\nu}}$.
Thus, if we set
$\mathcal{E}^G:=\mathcal{D}^G+\boldsymbol{\delta}$, we have 

\begin{equation}
\label{eqn:pw}
H(X)=\bigoplus_{\boldsymbol{\nu}\in \mathcal{E}^G}H(X)^{\tilde{\mu}}_{\boldsymbol{\nu}}
\end{equation}
where $H(X)^{\tilde{\mu}}_{\boldsymbol{\nu}}$ is the isotypical component corresponding to $V_{\boldsymbol{\nu}}$.
For each $\boldsymbol{\nu}\in \mathcal{E}^G$ we have the associated equivariant Szeg\"{o} projector
$\Pi_{\boldsymbol{\nu}}^{\tilde{\mu}}:L^2(X)\rightarrow H(X)^{\tilde{\mu}}_{\boldsymbol{\nu}}$.

In general, $H(X)^{\tilde{\mu}}_{\boldsymbol{\nu}}$ may well be infinite-dimensional, and
does not correspond to a space of holomorphic sections of any tensor power of $A$. Nonetheless,
it follows
from the theory of \cite{gs2} that if $\mathbf{0}\not\in \Phi(M)$, then
$\dim  H(X)^{\tilde{\mu}}_{\boldsymbol{\nu}} <+\infty$ for every $\boldsymbol{\nu}$
(see \S 2 \cite{pao-IJM}). Thus $\Pi_{\boldsymbol{\nu}}^{\tilde{\mu}}$ is a smoothing operator, so that 
its distributional kernel $\Pi_{\boldsymbol{\nu}}^{\tilde{\mu}}\in \mathcal{C}^\infty(
X\times X)$.

We are interested in the local asymptotics of 
$\Pi_{k\,\boldsymbol{\nu}}^{\tilde{\mu}}$ for a fixed $\boldsymbol{\nu}\in \mathcal{E}^G$ 
and $k\rightarrow+\infty$
with $k\,\boldsymbol{\nu}\in \mathcal{E}^G$. The latter condition is satisfied for any $k$ if
$\boldsymbol{\delta}\in \mathcal{D}^G$, so that 
$\mathcal{E}^G\subset \mathcal{D}^G$, as is the case when $G$ is simply connected. 

This general theme has already been studied in specific cases (\cite{pao-IJM}, \cite{pao-loa}, 
\cite{ca}, \cite{gp1}, \cite{gp2}),
building on the approach developed in \cite{z}, \cite{bsz} and \cite{sz} to the basic Fourier
case where $G=S^1$, $\mu$ is trivial 
and $\Phi=\imath$. We refer the reader to the introductions of 
\cite{pao-IJM}, \cite{pao-loa}, 
\cite{gp1}, \cite{gp2} for an ampler discussion of motivation and general framing.
The theme is geometrically relevant, being related to
interesting geometric quotients \cite{pao-polorb}.

The results in \cite{gp1} and \cite{gp2} are based on the pairing of the Weyl character and 
integration formulae with the techniques in \cite{z} and \cite{sz}.
The new ingredient here is the Kirillov character formula (\cite{ki}, \cite{r}) which considerably simplifies 
some of the
arguments, and allows to deal with more general Lie groups;
on the other hand, it forces restrictions on the stabilizer subgroups.

The following results are governed by the interplay between 
$\Phi$ and the cone over the coadjoint orbit through $\boldsymbol{\nu}$,
$\mathcal{O}_{\boldsymbol{\nu}}\subset \mathfrak{g}^\vee$. As $\boldsymbol{\nu}$ is 
a regular element of $\mathfrak{g}^\vee$, 
$\mathcal{O}_{\boldsymbol{\nu}}$ is equivariantly diffeomorphic to 
$G/T$, hence it  has dimension $d_G-r_G$. 
Let us set $\mathcal{C}( \mathcal{O}_{\boldsymbol{\nu}}):=\mathbb{R}_+\cdot 
\mathcal{O}_{\boldsymbol{\nu}}$.

We shall need the following hypothesis.
 
\begin{asmpt}
\label{asmpt:basic1}
We shall assume that:
\begin{enumerate}
\item $\mathbf{0}\not\in \Phi(m)$;
\item $\Phi$ is transverse to $\mathcal{C}( \mathcal{O}_{\boldsymbol{\nu}})$
(equivalently, $\Phi$ is transverse to the ray $\mathbb{R}_+\cdot \boldsymbol{\nu}$);
\item $M_{\mathcal{O}_{\boldsymbol{\nu}}}
:=\Phi^{-1}\big( \mathcal{C}( \mathcal{O}_{\boldsymbol{\nu}})\big)\neq \emptyset$.

\end{enumerate}

\end{asmpt}

If Assumption \ref{asmpt:basic1} holds, $M_{\mathcal{O}_{\boldsymbol{\nu}}}$ is a compact and 
connected $G$-invariant submanifold of
$M$, of (real) codimension $r_G-1$ (see the discussion of \cite{gp1}).
Let us set
$X_{\mathcal{O}_{\boldsymbol{\nu}}}:=\pi^{-1}(M_{\mathcal{O}_{\boldsymbol{\nu}}})$.

We shall also make the following assumption on the compact and connected Lie group $G$.

\begin{asmpt}
\label{asmpt:basic2}
Let $L(G)\subset \mathfrak{t}^\vee$ be the lattice of integral forms
on $G$; then
$\boldsymbol{\delta}\in L(G)$. 
\end{asmpt}

This condition is satisfied
if $G$ is either $U(n)$ for some $n\ge 1$, or 
a connected and simply connected compact semisimple Lie group.
If $G$ satisfies this assumption, then it is called \textit{acceptable}
in Harish-Chandra's terminology
(\S 2.5 of \cite{v1}).
Under Assumption \ref{asmpt:basic2}, $\mathcal{E}^G\subset \mathcal{D}^G$.

In the following we shall assume throughout that 
Assumptions \ref{asmpt:basic1} and \ref{asmpt:basic2}
hold.

\begin{thm}
\label{thm:rapid decrease resc}
Suppose that
$\mathcal{O}_{\boldsymbol{\nu}}\cap \mathfrak{t}^0=\emptyset$.
Fix $C,\,\epsilon>0$. Then, uniformly for
$\mathrm{dist}_X\left(G\cdot x,G\cdot y  \right)\ge C\,k^{\epsilon-\frac{1}{2}}$,
we have
$$
\Pi^{\tilde{\mu}}_{k\,\boldsymbol{\nu}}(x,y)=O\left(k^{-\infty}\right).
$$
\end{thm}

When
$G=U(n)$, the previous hypothesis is satisfied by any $\boldsymbol{\nu}\in \mathcal{E}^G$
with $\sum_{j=1}^n\nu_j\neq 0$.
It is never satisfied when $G=SU(2)$,
but the statement of Theorem \ref{thm:rapid decrease resc} is 
nonetheless true in this case, 
see \cite{gp2}.
More generally, let $\mu^T:T\times M\rightarrow M$ be the restriction of $\mu$, and
let $\Phi^T:M\rightarrow \mathfrak{t}^\vee$ be the moment map induced by
$\Phi$ (that is, the composition of $\Phi$ with the restriction
$\mathfrak{g}^\vee \rightarrow \mathfrak{t}^\vee$).
If $\mathbf{0}\not\in \Phi^T(M)$, then the hypothesis of Theorem 
\ref{thm:rapid decrease resc}
is satisfied for any $\boldsymbol{\nu}$ such that
$M_{\mathcal{O}_{\boldsymbol{\nu}}}\neq \emptyset$.

\begin{thm}
\label{thm:rapid decrease off locus resc}
Let us fix $C,\,\epsilon>0$, and assume that $\tilde{\mu}$ is free along
$X_{\mathcal{O}}$. Then, uniformly for 
$$
\max\left\{\mathrm{dist}_X(x,X_{\mathcal{O}_{\boldsymbol{\nu}}}),\mathrm{dist}_X(y,X_{\mathcal{O}_{\boldsymbol{\nu}}})\right\}
\ge C\,k^{\epsilon-\frac{1}{2}},
$$
we have 
$$
\Pi^{\tilde{\mu}}_{k\,\boldsymbol{\nu}}(x,y)=O\left(k^{-\infty}\right).
$$
If $\tilde{\mu}$ is only generically free along $X_{\mathcal{O}}$, and 
$X'\subset X$ is the open subset where it is free, the same estimate holds 
uniformly on compact subsets of $X'$.
\end{thm}

We shall focus on the near diagonal asymptotics of $\Pi^{\tilde{\mu}}_{k\,\boldsymbol{\nu}}(x,x)$
for $x$ belonging to a shrinking tubular neighborhood of $X_{\mathcal{O}_{\boldsymbol{\nu}}}$, of radius
$O\left(k^{\epsilon-1/2}\right)$. Using the normal exponential map, we may parametrize
such a neighborhood by a neighborhood of the zero section in the normal bundle of
$X_{\mathcal{O}_{\boldsymbol{\nu}}}\subset X$, which is the pull-back of the normal bundle 
$N(M_{\mathcal{O}_{\boldsymbol{\nu}}}/M)$ of
$M_{\mathcal{O}_{\boldsymbol{\nu}}}\subset M$.
If $V_{\mathcal{O}_{\boldsymbol{\nu}}}\subset X$ is a sufficiently small neighborhood of 
$X_{\mathcal{O}_{\boldsymbol{\nu}}}$, we shall accordingly write the general $y\in V_{\mathcal{O}_{\boldsymbol{\nu}}}$ in additive notation as $y=x+\mathbf{v}$, for unique
$x\in X_{\mathcal{O}_{\boldsymbol{\nu}}}$ and $\mathbf{v}\in 
N_{\pi (x)}(M_{\mathcal{O}_{\boldsymbol{\nu}}}/M)$\footnote{We may interpret $x+\mathbf{v}$ in terms of a system of 
Hesenberg local coordinates on $X$ centered at $x$ \cite{sz}, smoothly varying with $x$, see
\S \ref{sctn:proofrapiddec}.}.

In order to state the next Theorem, some further notation is needed. 

\begin{defn}
\label{defn:leading order coeff}
Let $\varphi$ be an $\mathrm{Ad}$-invariant Euclidean product on $\mathfrak{g}$, 
with associated norm $\|\cdot\|_\varphi$.
Let us also denote by $\varphi$ the induced bi-invariant Riemannian metric on $G$.
Clearly, $\varphi$ restricts to an invariant Riemannian metric $\varphi^T$ on $T$.
We shall adopt the following notation:

\begin{enumerate}
\item $\mathrm{d}^\varphi V_G$: the Riemannian density on $G$ associated to $\varphi$;
\item $\mathrm{vol}^\varphi (G)=\int_G\,\mathrm{d}^\varphi V_G(g)$; 
\item $\mathrm{d}^\varphi V_T$: the Riemannian density on $T$ 
\item $\mathrm{vol}^\varphi (T):=\int_T\,\mathrm{d}^\varphi V_T(t)$.
\item For any $\boldsymbol{\gamma}\in \mathfrak{g}^\vee$,
\begin{itemize}
\item $\boldsymbol{\gamma}^\varphi\in \mathfrak{g}$ is uniquely
determined by the condition $\boldsymbol{\gamma}=\varphi\big(\boldsymbol{\gamma}^\varphi,\cdot)$;
\item $\|\boldsymbol{\gamma}\|_\varphi:=\|\boldsymbol{\gamma}^\varphi\|_\varphi$;
\item $\boldsymbol{\gamma}_{\varphi,u}:=\frac{1}{•\|\boldsymbol{\gamma}^{\varphi}\|_\varphi}\,
\boldsymbol{\gamma}
=\frac{1}{•\|\boldsymbol{\gamma}\|_\varphi}\,
\boldsymbol{\gamma}
\in \mathfrak{g}^\vee$;
\item $\boldsymbol{\gamma}^{\varphi}_{u}:=\frac{1}{•\|\boldsymbol{\gamma}^{\varphi}\|_\varphi}\,
\boldsymbol{\gamma}^{\varphi}\in \mathfrak{g}$.
\end{itemize} 
\item $\mathfrak{t}^{\perp_{\varphi}}\subseteq \mathfrak{g}$: the Euclidean orthocomplement of
$\mathfrak{t}$ with respect to $\varphi$.
\item  For any $\boldsymbol{\tau}\in \mathfrak{t}$, 
$S_{\boldsymbol{\tau}}:\mathfrak{t}^{\perp_\varphi}\rightarrow \mathfrak{t}^{\perp_\varphi}$
denotes the restriction of $\mathrm{ad}_{\boldsymbol{\tau}}$;
when $\boldsymbol{\tau}$ is \textit{regular} 
(as is the case when $\boldsymbol{\tau}=\boldsymbol{\nu}^\varphi$ for
$\boldsymbol{\nu}\in \mathcal{E}^G$),
$S_{\boldsymbol{\tau}}$ is a linear automorphism, skew-symmetric with respect to
the restriction of $\varphi$.
\end{enumerate}
\end{defn}

Let us identify the coalgebra of $T$, $\mathfrak{t}^\vee$, with the subspace
of those $\boldsymbol{\lambda}\in \mathfrak{g}^\vee$ fixed by $T$ under the coadjoint action,
and let $\mathfrak{t}^\vee_{reg}\subset \mathfrak{t}^\vee$ be the open and dense subset
of those elements, called regular, that are fixed precisely by $T$.
Hence $\mathcal{E}^G\subset \mathfrak{t}^\vee_{reg}$.

By definition, for any $m\in M_{\mathcal{O}_{\boldsymbol{\nu}}}$
there exist $h_m\,T\in G/T$ and $\varsigma (m)>0$ such that
\begin{equation}
\label{eqn:hmzetam}
\Phi (m)=\varsigma (m)\,\mathrm{Coad}_{h_m}(\boldsymbol{\nu})\in 
\mathrm{Coad}_{h_m}(\mathfrak{t}^\vee).
\end{equation}
Since $\boldsymbol{\nu}$ is regular, $h_m\,T$ and $\varsigma (m)$ are uniquely
determined, and the functions 
$m\in M_{\mathcal{O}_{\boldsymbol{\nu}}}\mapsto h_m\,T\in G/T$ and
$\varsigma: M_{\mathcal{O}_{\boldsymbol{\nu}}}\rightarrow \mathbb{R}$
are smooth.

For any non-zero $\boldsymbol{\lambda}\in \mathfrak{g}^\vee$, let 
$\boldsymbol{\lambda}^0=\left( \boldsymbol{\lambda}^\varphi  \right)^{\perp_\varphi}
\subset \mathfrak{g}$ be its annihilator hyperplane.
Then
\begin{equation}
\label{eqn:annihilator phim}
\Phi(m)^0=\mathrm{Ad}_{h_m}\left(\boldsymbol{\nu}^0\right).
\end{equation}
Let us set $\mathfrak{t}_{\boldsymbol{\nu}}:=\mathfrak{t}\cap \boldsymbol{\nu}^0$,
so that we have a $\varphi$-orthogonal direct sum decomposition
$$
\mathfrak{t}=\mathrm{span}(\boldsymbol{\nu}^\varphi)\oplus \mathfrak{t}_{\boldsymbol{\nu}}.
$$
For every $m\in M_{\mathcal{O}_{\boldsymbol{\nu}}}$, we shall set
$\mathfrak{t}_m:=\mathrm{Ad}_{h_m}(\mathfrak{t})$; thus
$\mathfrak{t}_m$ is the unique Cartan subalgebra of $\mathfrak{g}$
containing $\Phi(m)$, or equivalently the Lie subalgebra of the (unique) maximal
torus $T_m=h_m\,T\,h_m^{-1}$ stabilizing $\Phi(m)$. Thus,
\begin{equation}
\label{eqn:t_m characterized}
\mathfrak{t}_m=\left\{ \boldsymbol{\eta}\in 
\mathfrak{g}:\,\left[\boldsymbol{\eta},\Phi(m)^\varphi\right]=0\right\}.
\end{equation}
Furthermore, we shall set
$\mathfrak{t}_m':=\mathrm{Ad}_{h_m}(\mathfrak{t}_{\boldsymbol{\nu}})=
\mathfrak{t}_m\cap \Phi (m)^0$. More explicitly,
\begin{equation}
\label{eqn:t_m' characterized}
\mathfrak{t}_m'=\left\{ \boldsymbol{\eta}\in 
\mathfrak{g}:\,\left[\boldsymbol{\eta},\Phi(m)^\varphi\right]=0,
\,\langle\Phi(m),\boldsymbol{\eta}\rangle =0\right\}.
\end{equation}
Hence we have the $\varphi$-orthogonal direct sum
\begin{equation}
\label{eqn:tprime m}
\mathfrak{t}_m=\mathrm{span}\big(\Phi(m)^\varphi)\oplus \mathfrak{t}_m'.
\end{equation}

Assume $m\in M_{\mathcal{O}_{\boldsymbol{\nu}}}$. 
Then $\mathrm{val}_m$ is injective on $\Phi(m)^0$ by Remark \ref{rem:equiv transv lf};
hence $\mathrm{val}_m$ is injective \textit{a fortiori} on $\mathfrak{t}_m'$.

\begin{defn}
\label{defn:Dphim}
In the following, $m\in M_{\mathcal{O}_{\boldsymbol{\nu}}}$. 
We shall adopt the following notation.
\begin{enumerate} 
\item $\mathrm{val}_m^*(\rho^M_m)$: the pull-back to
$\mathfrak{g}$ of the Euclidear product $\rho^M_m=\omega_m(\cdot,J_m\cdot)$ on $T_mM$;
$\rho'_m$: the restriction of 
$\mathrm{val}_m^*(\rho^M_m)$
to $\mathfrak{t}_m'$. Thus $\rho'_m$ is
non-degenerate (whence positive definite).

\item Given an arbitrary orthonormal basis $\mathcal{R}_m$ of 
$\mathfrak{t}_m'$ for the restriction of $\varphi$, let $D^\varphi(m):=M_{\mathcal{R}_m}(\rho'_m)$ be the representative
matrix of $\rho'_m$ w.r.t. $\mathcal{R}_m$, and set
$$
\mathcal{D}^\varphi (m):=\sqrt{\det D^\varphi(m)};
$$
then $\mathcal{D}^\varphi:M_{\mathcal{O}_{\boldsymbol{\nu}}}\rightarrow (0,+\infty)$ is well-defined 
and $\mathcal{C}^\infty$;
\end{enumerate}
We can now define a $\mathcal{C}^\infty$ function 
$\Psi_{\boldsymbol{\nu}}:M_{\mathcal{O}_{\boldsymbol{\nu}}}\rightarrow \mathbb{R}_+$ by setting
$$
\Psi_{\boldsymbol{\nu}}(m):=2^{1+\frac{r_G-1}{2}}\,\pi\cdot 
\frac{1}{\|\Phi (m)\|_{\varphi}\,\mathcal{D}^\varphi (m)}\cdot 
\frac{\mathrm{vol}(\mathcal{O}_{\boldsymbol{\nu}_{\varphi,u}},\sigma_{\boldsymbol{\nu}_{\varphi,u}})^2}
{\big| \det(S_{\boldsymbol{\nu}^{\varphi}_u})  \big|}
\cdot 
\frac{\mathrm{vol}^\varphi(T)}{\mathrm{vol}^\varphi(G)^2}.
$$
\end{defn}
By Theorem \ref{thm:scal asymp}
below, $\Psi_{\boldsymbol{\nu}}$ is actually independent of the choice of $\varphi$.

We need some further pieces of notation.

Given a real vector subspace $R\subseteq T_mM$, we shall
denote by $R^{\perp_{h_m}}$ the orthocomplement of $R$ with respect to the Hermitian structure
$h_m=\rho^M_m-\imath\,\omega_m$; equivalently, $R^{\perp_{h_m}}$ is the orthocomplement of the complex
subspace $R+J_m(R)$ of $T_mM$, and is a complex subspace of $(T_mM,J_m)$. Clearly, 
\begin{equation}
\label{eqn:Rhsomega}
R^{\perp_{h_m}}=R^{\perp_{\rho^M_m}}\cap R^{\perp_{\omega_m}}.
\end{equation}
where $R^{\perp_{\rho^M_m}}$ and $R^{\perp_{\omega_m}}$ are, respectively, the Riemannian and symplectic
orthocomplements of $R$. 

If $m\in M$, let $\mathfrak{g}_M(m)\subseteq T_mM$ be the vector subspace given by
the evaluations at $m$ of the all the vector fields on $M$ induced by the elements of
$\mathfrak{g}$. 
We shall see in Lemma \ref{lem:normal bundle} that
for any $m\in M_{\mathcal{O}_{\boldsymbol{\nu}}}$ the normal space of
$M_{\mathcal{O}_{\boldsymbol{\nu}}}$ at $m$ satisfies
\begin{equation}
\label{eqn:normal space Mm}
N_m(M_{\mathcal{O}_{\boldsymbol{\nu}}})\subseteq J_m\big( \mathfrak{g}_M(m) \big),\quad 
\text{hence}\quad
N_m(M_{\mathcal{O}_{\boldsymbol{\nu}}})\cap \mathfrak{g}_M(m)^{\perp_{h_m}}
=(0).
\end{equation}

In this setting, small displacements from a fixed $x\in X$ are conveniently
expressed in Heisenberg local coordinates (HLC) on $X$ centered at $x$ \cite{sz}.
A choice of HLC at $x$ gives a meaning to the expression $x+\mathbf{v}$, where 
$\mathbf{v}\in T_{\pi(x)}M$ has sufficiently small norm. Furthermore, the curve
$\gamma_{x,\mathbf{v}}:\tau\in (-\epsilon,\epsilon)\mapsto x+\tau\,\mathbf{v}$ is horizontal at
$\tau=0$, and in fact $\gamma_{x,\mathbf{v}}'(0)=\mathbf{v}^\sharp$.
More will be said in \S \ref{sctn:proofrapiddec}.

A further notational ingredient that will go into the statement of Theorem \ref{thm:scal asymp} is
an invariant governing the exponential decay of various asymptotics related to Szeg\"{o}
kernels \cite{sz}. 

\begin{defn}
\label{defn:psi2}
Let $\|\cdot\|$ and $\omega_0$ be the standard norm and symplectic structures
on $\mathbb{R}^{2\,d}$, respectively. Let us define
$\psi_2:\mathbb{R}^{2\,d}\times \mathbb{R}^{2\,d}\rightarrow \mathbb{R}$ by setting
$$
\psi_2(u,v):=-\imath\,\omega_0(u,v)-\frac{1}{2•}\,\|u-v\|^2.
$$
\end{defn}

A choice of Heisenberg local coordinates at $x$ entails the choice of
a unitary ismophism $T_{\pi(x)}M\cong \mathbb{C}^d$ (with the standard Hermitian structure),
by means of which we shall view $\psi_2$ as being defined on $T_{\pi(x)}M$.
For the sake of simplicity, we shall consider displacements of the
form $x+k^{-1/2}\,(\mathbf{v}+\mathbf{w})$, where $\mathbf{v}$ is normal to
$M_{\mathcal{O}_{\boldsymbol{\nu}}}$ and $\mathbf{w}$ is in $\mathfrak{g}_M(m)^{\perp_{h_m}}$.
Heuristically, $x+\mathbf{w}$ covers a displacement in a suitable quotient.

\begin{thm}
\label{thm:scal asymp}
Assume that $x\in X_{\mathcal{O}_{\boldsymbol{\nu}}}$, and that $\tilde{\mu}$ is free at
$x$. Set $m_x=\pi (x)$ and fix $C>0$, $\epsilon\in (0,1/6)$.
Then, uniformly for 
$\mathbf{v}_j\in N_{m_x}(M_{\mathcal{O}_{\boldsymbol{\nu}}})$ and 
$\mathbf{w}_j\in \mathfrak{g}_M(m)^{\perp_{h_m}}$ satisfying
$\|\mathbf{v}_j\|,\,\|\mathbf{w}_j\|\le C\,k^\epsilon$, the following
asymptotic expansion holds as $k\rightarrow +\infty$:
\begin{eqnarray*}
\lefteqn{\Pi^{\tilde{\mu}}_{k\,\boldsymbol{\nu}}
\left(x+\frac{1}{\sqrt{k}}\,(\mathbf{v}_1+\mathbf{w}_1) ,
x+\frac{1}{\sqrt{k}}\,(\mathbf{v}_2+\mathbf{w}_2)   \right)}\\
&\sim&
\Psi_{\boldsymbol{\nu}}(m_x)\,\left(\frac{k}{\varsigma(m_x)\,\pi}\right)^{d+\frac{1-r_G}{2}}
\cdot e^{\frac{1}{\varsigma(m_x)}\,\,\left[ 
 \psi_2(\mathbf{w}_1,\mathbf{w}_2)  -(\|\mathbf{v}_1\|_{m_x}^2+\|\mathbf{v}_2\|_{m_x}^2 )
\right]}\\
&&\cdot \left[1+\sum_{j\ge 1}k^{-j/2}\,P_j(m_x;\mathbf{v}_j,\mathbf{w}_j)   \right],
\end{eqnarray*}
where $P_j(m_x;\cdot)$ is a polynomial of degree $\le 3\,j$ and parity $j$.
If $X'_{\mathcal{O}_{\boldsymbol{\nu}}}\subseteq X_{\mathcal{O}_{\boldsymbol{\nu}}}$
is the open subset on which $\tilde{\mu}$ is free, the estimate holds uniformly on 
the compact subsets of $X'_{\mathcal{O}_{\boldsymbol{\nu}}}$. 
\end{thm}

By a Gaussian integral computation in normal Heisenberg coordinates, as in the
proof of Corollary 1.3 of \cite{pao-IJM}, one can then deduce the following:

\begin{cor}
Assume that $\tilde{\mu}$ is free along $X_{\mathcal{O}_{\boldsymbol{\nu}}}$. Then
there is an asymptotic expansion
$$
\dim H(X)^{\tilde{\mu}}_{k\,\boldsymbol{\nu}}\sim 
\left(\frac{k}{\pi}\right)^{d+1-r_G}\,\left[
\delta_{\boldsymbol{\nu},0}+k^{-1}\,\delta_{\boldsymbol{\nu},1}+\ldots\right],
$$
with
$$
\delta_{\boldsymbol{\nu},0}:=\frac{1}{2^{\frac{r_g-1}{2}}}\,\int_{M_{\mathcal{O}_{\boldsymbol{\nu}}}}
\left[ \frac{\Psi_{\boldsymbol{\nu}}(m)}{\varsigma (m)^{d+1-r_G}} \right]\,\mathrm{d}V_{M_{\mathcal{O}_{\boldsymbol{\nu}}}}(m),
$$
where $\mathrm{d}V_{M_{\mathcal{O}_{\boldsymbol{\nu}}}}$
is the density on $M_{\mathcal{O}_{\boldsymbol{\nu}}}$ for the induced Riemannian metric.
\end{cor}

In closing this  introduction, we mention that there is a wider scope for the results in this
paper. While
our focus is on the complex projective setting,
in view of the microlocal theory of almost complex Szeg\"{o} kernels in \cite{sz}
the present approach can be naturally extended to the compact symplectic category.

\section{Examples}

We check the statement of Theorem \ref{thm:weyl cf} against those in
\cite{pao-IJM} , \cite{gp1}, \cite{gp2}. 

\begin{exmp}
Suppose $G=T$ is an $r$-dimensional torus.
Let us take the standard metric $\varphi$, so that
$\mathrm{vol}^\varphi(G)=(2\,\pi)^r$.
We obtain
$$
\Psi_{\boldsymbol{\nu}}(m_x)=\frac{2^{1+\frac{r-1}{2}}\,\pi}{(2\,\pi)^r}
\cdot 
\frac{1}{\|\Phi (m)\|_{\varphi}\,\mathcal{D}^\varphi (m)}
=\frac{1}{\left(\sqrt{2}\,\pi  \right)^{r-1}}\,\frac{1}{\|\Phi (m)\|_{\varphi}\,\mathcal{D}^\varphi (m)}.
$$
Thus Theorem \ref{thm:weyl cf} fits with Theorem 2 of \cite{pao-IJM}.
\end{exmp}

\begin{exmp}
Suppose $G=SU(2)$, so that $d_G=3$, $r_G=1$. Let 
$\varphi:\mathfrak{su}(2)\times \mathfrak{su}(2)\rightarrow \mathbb{R}$
be defined by 
$$
\varphi(A,B):=\mathrm{trace}\left(A\,\overline{B}^t\right)
=\mathrm{trace}\left(A\,B\right).
$$
Let $T\leqslant G$ be the standard torus; $\mathfrak{t}$ is generated by the
diagonal matrix $Z$ with entries $\imath,\,-\imath$, which has norm $\sqrt{2}$.
For $\nu\in \mathbb{Z}$, we shall denote by $\boldsymbol{\nu}\in \mathfrak{t}^\vee$
the weight taking value $\nu$ on $Z$.
We have
$$
\mathrm{vol}^\varphi(G)=2^{3/2}\cdot (2\,\pi^2),\quad 
\mathrm{vol}^\varphi(T)=\sqrt{2}\cdot  2\,\pi.
$$

For any $\nu\in \mathbb{Z}$, let $\boldsymbol{\nu}$ be the weight such that
$\langle\boldsymbol{\nu},Z\rangle =\nu$.
Then
$$
\boldsymbol{\nu}^\varphi=\frac{\nu}{2}\,Z,\quad \|\boldsymbol{\nu}\|_{\varphi}=\frac{1}{\sqrt{2}}\,\nu.
$$
Furthermore,
$$
\mathrm{vol}(\mathcal{O}_{\boldsymbol{\nu}},\sigma_{\boldsymbol{\nu}})=2\,\pi\,\nu,
\quad
|\det(S_{\boldsymbol{\nu}^\varphi})|=\nu^2.
$$

Finally, let $\lambda (m)>0$ be defined by the condition that $\Phi(m)^\varphi$ be similar to 
$\lambda (m)\,Z$.
Then $\|\Phi(m)\|_{\varphi}=\left\|\Phi(m)^\varphi\right\|_{\varphi}=\sqrt{2}\,\lambda (m)$
and $\varsigma (m)=\big(2\,\lambda (m)\big)/\nu$.
We obtain 
\begin{equation*}
\Psi_{\boldsymbol{\nu}}(m)=2\,\pi\cdot 
\frac{1}{\sqrt{2}\,\lambda (m)}\cdot 
4\,\pi^2
\cdot 
\frac{\sqrt{2}\cdot  2\,\pi}{\left( 2^{3/2}\cdot (2\,\pi^2) \right)^2}
=\frac{1}{2\,\lambda (m)},
\end{equation*}
in agreement with \cite{gp2}.
\end{exmp}

\begin{exmp}
If $G=U(2)$, we have $d_G=4$, $r_G=2$.
Let $\varphi:\mathfrak{u}(2)\times \mathfrak{u}(2)\rightarrow \mathbb{R}$ be defined
as for $SU(2)$. Let $T\leqslant G$ be the standard maximal torus; then $\mathfrak{t}$ has orthonormal
basis $(R,S)$, where $R$ and $S$ are the diagonal matrices with diagonal entries
$\begin{pmatrix}
\imath&0
\end{pmatrix}$ and $\begin{pmatrix}
0&\imath
\end{pmatrix}$, respectively.
Let $\boldsymbol{\nu}=\nu_1\,R^*+\nu_2\,S^*$, where $(R^*,S^*)$ is the dual basis.
Then $\boldsymbol{\nu}^\varphi=\nu_1\,R+\nu_2\,S$.

We have in this case
$$\mathrm{vol}^\varphi(\mathcal{O}_{\boldsymbol{\nu}},\sigma_{\boldsymbol{\nu}})
=2\,\pi\,(\nu_1-\nu_2), \quad |\det(S_{\boldsymbol{\nu}^\varphi})|=(\nu_1-\nu_2)^2.
$$
Furthermore,
$$
\mathrm{vol}^\varphi(T)=(2\,\pi)^2,\quad \mathrm{vol}^\varphi(G)=8\,\pi^3.
$$
We obtain for $m\in M_{\mathcal{O}_{\boldsymbol{\nu}}}$
\begin{equation}
\label{eqn:psinuu2}
\Psi_{\boldsymbol{\nu}}(m)=
2^{\frac{3}{2}}\,\pi\cdot 
\frac{1}{\|\Phi (m)\|_{\varphi}\,\mathcal{D}^\varphi (m)}\cdot 
4\,\pi^2
\cdot 
\frac{4\,\pi^2}{64\,\pi^6}=\frac{1}{\sqrt{2}\,\pi}\cdot 
\frac{1}{\|\Phi (m)\|_{\varphi}\,\mathcal{D}^\varphi (m)},
\end{equation}
which tallies with the front factor in the 
pointwise expansion in Theorem 1.4 of \cite{gp1}. 
In the latter expansion the numerical factor is written in a slightly 
less explicit form, but replacing $V_3=2\,\pi^2$ it is readily seen to equal
the one in (\ref{eqn:psinuu2}).
\end{exmp}

\section{Preliminaries}
\label{sctn:preliminaries}

We shall adopt the following notational conventions:
\begin{enumerate}
\item $R_j$ will denote a smooth real, complex or vector valued function defined in the
neighborhood of the origin of some vector space, vanishing to $j$-th order at the
origin, and allowed to vary from line to line;
\item if $G$ acts smoothly on a manifold $Z$ and $\boldsymbol{\xi}\in \mathfrak{g}$,
$\boldsymbol{\xi}_Z$ will denote the induced vector field on $Z$;
\item under the same assumption, if $p\in Z$ we shall denote by
$\mathrm{val}_p:\boldsymbol{\xi}\in\mathfrak{g}\rightarrow \boldsymbol{\xi}_Z(p)\in T_pZ$
the evaluation map;
\item if $m\in M$ and $\mathbf{v}\in T_mM$, we shall denote by $\|\mathbf{v}\|_m$
the norm of $\mathbf{v}$ with respect to $\rho^M$;
\item if $x\in X$ and $\upsilon=a\,\partial_\theta|_x+\mathbf{v}^\sharp\in T_xX$, 
in computations it will be convenient to set 
$\|\upsilon\|_x:=\sqrt{a^2+\|\mathbf{v}\|_m^2}$ (this is the norm in an obvious vertical
rescaling of $\rho^X$).
\end{enumerate}

\begin{rem}
\label{rem:equiv transv lf}
Arguing as in \S 2 of \cite{pao-IJM}  (or \S 4.1.1 of \cite{gp1}), one verifies
that the following conditions are equivalent:

\begin{enumerate}
\item Assumption \ref{asmpt:basic1} holds;
\item $\tilde{\mu}$ is locally free along $X_{\mathcal{O}_{\boldsymbol{\nu}}}$;
\item for every $m\in M_{\mathcal{O}_{\boldsymbol{\nu}}}$,
$\mathrm{val}_m:\mathfrak{g}\rightarrow  T_mM$
is injective on the annihilator of $\Phi(m)$, that is, 
$$
\ker(\mathrm{val}_m)\cap \Phi(m)^0=(0).
$$
\end{enumerate}
\end{rem}

Let us define
\begin{eqnarray}
\label{eqn:Znu}
\mathcal{Z}_{\boldsymbol{\nu}}:=\big\{(x,y)\in X_{\mathcal{O}}\times X_{\mathcal{O}}:
\,y\in G\cdot x\big\}.
\end{eqnarray}
Then $\mathcal{Z}_{\boldsymbol{\nu}}$ is a $G\times G$-invariant compact and
connected submanifold of $X\times X$.

\begin{thm}
\label{thm:rapid decrease}
Uniformly on compact subsets of $(X\times X)\setminus \mathcal{Z}_{\boldsymbol{\nu}}$,
we have 
$$
\Pi^{\tilde{\mu}}_{k\,\boldsymbol{\nu}}(x,y)=O\left(  k^{-\infty}\right).
$$
\end{thm}

\begin{proof}
The argument is a slight modification of the one in \S 3.1 and \S 3.2 of \cite{gp1},
based on the theory in \cite{gs1};
hence we shall be somewhat sketchy. The ladder Szeg\"{o} projector
$$
\Pi_L:=\bigoplus _{k=1}^{+\infty}\Pi^{\tilde{\mu}}_{k\,\boldsymbol{\nu}}:
L^2(X)\longrightarrow \bigoplus_{k=1}^{+\infty}H(X)^{\tilde{\mu}}_{k\,\boldsymbol{\nu}} 
$$
has a distributional kernel whose wave front satisfies 
$
\mathrm{WF}(\Pi_L)\subseteq \mathcal{Z}_{\boldsymbol{\nu}}$.

Let $K\Subset (X\times X)\setminus \mathcal{Z}_{\boldsymbol{\nu}}$.
Without loss, we may assume that $K$ is $G\times G$-invariant. There exists a 
$G\times G$-invariant smooth cut-off
function $\varrho\ge 0$ on $X\times X$, which is identically equal to $1$ on a 
neighborhood of $K$, and
vanishes identically on a neighborhood of $\mathcal{Z}_{\boldsymbol{\nu}}$.
Hence $\varrho \cdot \Pi_L\in \mathcal{C}^\infty(X\times X)$. 
Hence, we obtain a $\mathcal{C}^\infty$ function
$$
F:(g,x,y)\in G\times X\times X\mapsto (\varrho \cdot \Pi_L)\left(\tilde{\mu}_{g^{-1}}(x),y\right)
\in \mathbb{C}.
$$
We shall set $F_{x,y}:=F(\cdot,x,y):G\rightarrow \mathbb{C}$.

Let us denote $P_{k\,\boldsymbol{\nu}}:L^2(X)\rightarrow L^2(X)_{k\,\boldsymbol{\nu}}$
the projector. Hence, $\Pi_{k\,\boldsymbol{\nu}}^{\tilde{\mu}}=
P_{k\,\boldsymbol{\nu}}\circ \Pi_L^{\tilde{\mu}}$. 
If $\mathrm{d}^HV_G(g)$ is the Haar measure on $G$, this means that for $(x,y)\in X\times X$
\begin{eqnarray}
\label{eqn:smooth integrand K}
\Pi^{\tilde{\mu}}_{k\,\boldsymbol{\nu}}(x,y)&=&d_{k\,\boldsymbol{\nu}}\int_G
\,\overline{\chi_{k\,\boldsymbol{\nu}}(g)}\,\Pi_L
\left(\tilde{\mu}_{g^{-1}}(x),y\right)\,\mathrm{d}^HV_G(g).
\end{eqnarray}

If $(x,y)\in K$, therefore, 
\begin{eqnarray}
\label{eqn:smooth integrand K1}
\Pi^{\tilde{\mu}}_{k\,\boldsymbol{\nu}}(x,y)&=&d_{k\,\boldsymbol{\nu}}\cdot\int_G
\,\overline{\chi_{k\,\boldsymbol{\nu}}(g)}\,(\varrho\,\Pi_L)
\left(\tilde{\mu}_{g^{-1}}(x),y\right)\,\mathrm{d}^HV_G(g)\\
&=&d_{k\,\boldsymbol{\nu}}\cdot 
\mathrm{trace}\big(\mathcal{F}(F_{x,y})(k\,\boldsymbol{\nu}-\boldsymbol{\delta})\big),
\nonumber
\end{eqnarray}
where $\mathcal{F}$ denotes the Fourier transform on $G$ \cite{su}, viewed as a function
on $\mathcal{D}^G$.
Since $d_{k\,\boldsymbol{\nu}}\le C_{\boldsymbol{\nu}}\,k^{d_G-r_G}$,
it suffices to show that 
$\mathcal{F}(F_{x,y})(k\,\boldsymbol{\nu}-\boldsymbol{\delta})=
O\left(k^{-\infty}\right)$ in Hilbert-Schmidt norm for $k\rightarrow +\infty$.
To this end, we apply arguments in \S 1 of \cite{su}.

To begin with, for any $\boldsymbol{\lambda}\in \mathcal{D}^G\setminus\{\mathbf{0}\}$
we have (see eq. (1.21) of \cite{su})
\begin{equation}
\label{eqn:parceval1}
\left\| \mathcal{F}\left(F_{x,y}\right) (\boldsymbol{\lambda})   \right\|_{H S}^2\le 
\|\boldsymbol{\lambda}\|^{-4\,l}\,\left\| \mathcal{F}\left(\Delta_G^lF_{x,y}\right) (\boldsymbol{\lambda})   \right\|_{H S}^2;
\end{equation}
here $\Delta_G$ denotes the Laplacian (Casimir) operator on $G$, and 
$\left\| \cdot   \right\|$ is the Hilbert-Schmidt norm. Hence for $k\gg 0$
\begin{equation}
\label{eqn:parceval1.1}
\left\| \mathcal{F}\left(F_{x,y}\right) (k\,\boldsymbol{\nu}-\boldsymbol{\delta})   \right\|_{H S}^2\le 
2\,k^{-4\,l}\|\boldsymbol{\nu}\|^{-4\,l}\,\left\| \mathcal{F}\left(\Delta_G^lF_{x,y}\right) (k\,\boldsymbol{\nu}-\boldsymbol{\delta})   \right\|_{H S}^2;
\end{equation}
On the other hand, by the Parceval identity
(eq. (1.16) of \cite{su}) we also have
\begin{equation}
\label{eqn:parceval1.2}
\left\| \mathcal{F}\left(\Delta_G^lF_{x,y}\right) (\boldsymbol{\lambda})   \right\|^2_{HS}\le 
\frac{1}{d_{\boldsymbol{\lambda}+\boldsymbol{\delta}}}\,\|\Delta_G^lF_{x,y}\|^2_2\le 
\|\Delta_G^lF_{x,y}\|^2_2
\end{equation}
where 
$\|\cdot\|_2$ denotes the $L^2$-norm on $G$
($d_{\boldsymbol{\lambda}+\boldsymbol{\delta}}=d(\boldsymbol{\lambda})$ in the notation of
\cite{su}).

Furthermore, by compactness for any $l\ge 0$ we can find $C_l>0$ such that 
$\|\Delta_G^lF_{x,y}\|_2^2\le C_l$ for all $(x,y)\in X\times X$.
Hence, by the Parceval identity, for every $\boldsymbol{\lambda}\in \mathcal{D}^G$
we have
\begin{equation}
\label{eqn:parceval1.3}
\left\| \mathcal{F}\left(\Delta_G^lF_{x,y}\right) (\boldsymbol{\lambda})   \right\|_{HS}^2\le 
\frac{1}{d_{\boldsymbol{\lambda}+\boldsymbol{\delta}}}\,C_l\le C_l.
\end{equation}
Therefore, for $k\gg 0$ we have
\begin{equation}
\label{eqn:parceval1.4}
\left\| \mathcal{F}\left(F_{x,y}\right) (k\,\boldsymbol{\nu}-\boldsymbol{\delta})   \right\|_{H S}^2\le 
k^{-4\,l}\,C_{l}'
\end{equation}

\end{proof}

\section{Proof of Theorem \ref{thm:rapid decrease resc}}

\label{sctn:proofrapiddec}
In the proof of Theorem \ref{thm:rapid decrease resc} we shall use the 
Weyl integration and character formulae, which we  briefly recall below, 
referring e.g. to \cite{v1} (\S 2.3-2.5) and  \cite{v2} (\S 4.13 and 4.14)
for a detailed treatment.

Let $W$ 
denote the Weyl group of $(\mathfrak{g},\mathfrak{t})$; then $W$ naturally acts on $\mathfrak{t}^\vee$
preserving the root lattice $L(R)\subset \mathfrak{t}^\vee$.

Let $L(G)\subset \mathfrak{t}^\vee$ be the lattice of integral forms of $G$.
Every $\boldsymbol{\gamma}\in L(G)$ defines a character 
$E_{\boldsymbol{\gamma}}:T\rightarrow S^1$, and we may define 
$$
A_{\boldsymbol{\gamma}}:=\sum_{s\in W}\epsilon(s)\,E_{s(\boldsymbol{\gamma})}:T\rightarrow 
\mathbb{C},
$$
where $\epsilon(s)=\det(s)$ (here $s$ is viewed as a linear map $\mathfrak{t}\rightarrow\mathfrak{t}$).
In particular, since $\boldsymbol{\delta}\in  L(G)$ by Assumption \ref{asmpt:basic2},
 we may set
$\Delta:=A_{\boldsymbol{\delta}}$.

Let $\mathrm{d}^HV_G$, $\mathrm{d}^HV_T$, $\mathrm{d}^HV_{G/T}$ 
be the Haar measures on $G$, $T$
and $G/T$, respectively (\S 2.3 of \cite{v1}, \S 4.13 of \cite{v2}).
Then the following holds (Theorem 4.13.5 of \cite{v2}):

\begin{thm}
\label{thm:weyl if}
(Weyl Integration Formula)
Let us set 
For every $L^1$-function on $G$, 
$$
\int _G \,f(g)\,\mathrm{d}^HV_G(g)=\frac{1}{|W|}\cdot\int_T\,f_T(t)\,\big|\Delta (t)\big|^2\,\mathrm{d}^HV_T(t),
$$
where
$$
f_T(t):=\int_{G/T}\,f\left( g\,t\,g^{-1} \right)\,\mathrm{d}^HV_{GT}(gT).
$$
\end{thm}

Similarly, under Assumption \ref{asmpt:basic2}
$A_{\boldsymbol{\nu}}$ is well-defined for any $\boldsymbol{\nu}\in \mathcal{E}^G$.
Then we have (Theorem 4.14.4 of \cite{v2}):
\begin{thm}
\label{thm:weyl cf}
(The Weyl Character formula)
On the open and dense regular locus $T'\subset T$ (defined by $\Delta\neq 0$), we have
$$
\left.\chi_{\boldsymbol{\nu}}\right|_{T'}=\frac{A_{\boldsymbol{\nu}}}{•\Delta}. 
$$
\end{thm}

Another basic ingredient of the following arguments
is the microlocal representation of $\Pi$ as a Fourier integral operator introduced
in \cite{bs}, and its elaboration in the projective (and sympletic) setting 
in \cite{z}, \cite{bsz}, \cite{sz}. We refer to the latter papers for a detailed discussion, and
simply recall that the latter description has the form
\begin{equation}
\label{eqn:szego oif}
\Pi (x,y)=\int_0^{+\infty}\,e^{\imath\,u\,\psi (x,y)}\,s(x,y,u)\,\mathrm{d}u,
\end{equation}
where $\psi$ is a complex phase function (with $\Im (\psi)\ge 0$) and 
$s$ is a semiclassical symbol on $X\times X$, admitting an asymptotic expansion
of the form 
\begin{equation}
\label{eqn:amplitude szego expanded}
s(x,y,u)\sim \sum_{j\ge 0}\,s_j(x,y)\,u^{d-j}. 
\end{equation}
We shall invoke the
following two properties of $\psi$:
\begin{enumerate}
\item for any $x\in X$, we have 
\begin{equation}
\label{eqn:diagodiffpsi}
\mathrm{d}_{(x,x)}\psi =(\alpha_x,-\alpha_x);
\end{equation}

\item there exists a constant $D_X>0$ such that ($\Im$ denoting imaginary part)
\begin{equation}
\label{eqn:imaginary part}
\Im \big( \psi (x,y) \big)\ge D_X\,\mathrm{dist}_X(x,y)^2\qquad \forall\,x,y\in X.
\end{equation}
\end{enumerate}

A further useful tool is given by Heisenberg local coordinates on $X$; these are
the local coordinates on $X$, introduced in \cite{sz}, in which the local scaling asymptotics
of Szeg\"{o} kernels exhibit their universal character. 
Given $x\in X$, a system of local Heisenberg coordinates on $X$ centered at $x$
will be denoted additively, in the form
$x+(\theta,\mathbf{v})$; here $\theta\in (-\pi,\pi)$ and
$\mathbf{v}\in 
\mathbb{C}^d$ varies in an open ball centered at the origin. 
Fiber rotation is represented by translation in $\theta$, and
the locus $\theta=0$ projects diffeomorphically onto its image in $M$,
hence defines a local section of $X$ which is horizontal at $x$ with respect to the
connection form $\alpha$.
Thus a system of Heisenberg local coordinates for $X$
centered at $x$ entails a choice of local coordinates for $M$ centered
at $m_x=\pi(x)$; actually the latter system determines a unitary isomorphism
$\mathbb{C}^d\cong T_mM$ with respect to $(\omega_m,J_m)$, so when writing
$x+(\theta,\mathbf{v})$ it is often assumed that $\mathbf{v}\in T_mM$
(as will be the case below).
Further, when $\theta=0$ we usually write $x+\mathbf{v}$ for
$x+(0,\mathbf{v})$ (as in the statement of Theorem \ref{thm:scal asymp}).
Referring to (\ref{eqn:amplitude szego expanded}), in Heisenberg local coordinates at $x$
we have 
\begin{equation}
\label{eqn:s_0heis}
s_0(x,x)=\frac{1}{\pi^d}.
\end{equation}
We refer the reader to \cite{sz} for a detailed discussion.

\begin{proof}[Proof of Theorem \ref{thm:rapid decrease resc}]
The proof is an adaptation of the argument used when $G=U(2)$ in \cite{gp1},
so we'll be somewhat sketchy.

To begin with, in view of Theorem \ref{thm:rapid decrease} we may assume without loss
that $x$ and $y$ belong to a small $S^1\times G$-invariant neighborhood $V_{\mathcal{O}}$
of $X_{\mathcal{O}}$.
Equivalently, $m_x:=\pi(x)$ and $m_y:=\pi(y)$ belong to a small $G$-invariant neighborhood of
$U_{\mathcal{O}}$ of $M_{\mathcal{O}}$ in $M$, and .
In particular, we may assume that $\tilde{\mu}$ is free on $V_{\mathcal{O}}$. 

Hence
$\Phi_G(m_x)$ belongs a small conic neighborhood of $\mathcal{C}(\mathcal{O}_{\boldsymbol{\nu}})$. 
Furthermore,
replacing $(x,y)$ by $\left(\tilde{\mu}_h(x),\tilde{\mu}_h(y)\right)$
for a suitable $h\in G$, we may also assume that $\Phi_G(m_y)$ 
belongs to a small conic neighborhood
of $\mathbb{R}_+\cdot \boldsymbol{\nu}$.
More precisely, we may assume that
$\Phi_G(m_y)=\lambda_y\,\boldsymbol{\nu}+\beta_y$, 
where $\lambda_y>0$,
$\beta_y\in\boldsymbol{\nu}^0$ are smooth
(here $\boldsymbol{\nu}^\perp\subset\mathfrak{g}^\vee$ is the orthcomplement of 
$\boldsymbol{\nu}$ with respect to, say, $\varphi^H$), and $\|\beta_y\|\ll \lambda_y$.

We have
\begin{equation}
\label{eqn:szego equiv proj}
\Pi_{k\,\boldsymbol{\nu}}(x,y)=
d_{k\,\boldsymbol{\nu}}\,\int_G \overline{\chi_{k\,\boldsymbol{\nu}}(g)}\,\Pi\left(
\tilde{\mu}_{g^{-1}}(x),y\right)\,\mathrm{d}^HV_G(g).
\end{equation}
%
For a suitably small $\delta_K>0$, let us introduce the open cover
$\mathcal{V}=\{V',V^{''}\}$ of
$G\times X\times X$ given by
\begin{eqnarray*}
V'&:=&\left\{ (g,x,y)\in G\times X\times X\,:\,\mathrm{dist}_X 
\left(\tilde{\mu}_{g^{-1}}(x),y   \right) <2\,\delta_K  \right\},\\
V^{''}&:=&\left\{ (g,x,y)\in G\times X\times X\,:\,\mathrm{dist}_X 
\left(\tilde{\mu}_{g^{-1}}(x),y   \right)>\delta_K  \right\}.
\end{eqnarray*}
Let $\varrho'+\varrho''=1$ be a partition of unity on $G\times X\times X$
subordinate to $\mathcal{V}$.
We then have 
\begin{equation}
\label{eqn:szego equiv proj dec}
\Pi_{k\,\boldsymbol{\nu}}(x,y)=
\Pi_{k\,\boldsymbol{\nu}}(x,y)'+\Pi_{k\,\boldsymbol{\nu}}(x,y)^{''},
\end{equation}
where $\Pi_{k\,\boldsymbol{\nu}}(x,y)'$ and $\Pi_{k\,\boldsymbol{\nu}}(x,y)^{''}$ are
defined as in (\ref{eqn:szego equiv proj}), except that the integrand has been multiplied by
$\varrho'$ and $\varrho^{''}$, respectively.
Therefore, $\mathcal{G}(g,x',y'):=\varrho^{''}\cdot \Pi\left(\tilde{\mu}_{g^{-1}}(x'),y'   \right)$
is a smooth function on $G\times X$. Hence, letting $\mathcal{G}_{x',y'}:=\mathcal{G}(\cdot,x',y')$
arguing as above we have
$$
\Pi_{k\,\boldsymbol{\nu}}(x',y')^{''}=d_{k\,\boldsymbol{\nu}}\,
\mathrm{trace}\big(\mathcal{F}(\mathcal{G}_{x',y'})(k\,\boldsymbol{\nu}-\boldsymbol{\delta})\big)
=O\left(k^{-\infty}\right),
$$
uniformly in $(x',y')\in X\times X$.
We are thus reduced to considering the asymptotics of 
the former summand $\Pi_{k\,\boldsymbol{\nu}}(x,y)'$.
On the support of $\varrho'$, $y$ belongs to a small neighborhood of
$\tilde{\mu}_{g^{-1}}(x)$; hence we may represent $\Pi$ as a Fourier integral operator.

Applying Theorems \ref{thm:weyl cf} and \ref{thm:weyl if}, 
and recalling that $\Delta\circ \sigma=\epsilon(\sigma)\,\Delta$ for any $\sigma\in W$,
we obtain 
\begin{eqnarray}
\label{eqn:szego weyl}
\Pi_{k\,\boldsymbol{\nu}}(x,y)&\sim&
 d_{k\,\boldsymbol{\nu}}\,\int_G\,\mathrm{d}^HV_G(g)\,\left[ 
\overline{\chi_{k\,\boldsymbol{\nu}}(g)}\,\varrho'(g,x,y)\,\Pi\left(\tilde{\mu}_{g^{-1}}(x),y  \right)  
\right]\\
&=&\frac{d_{k\,\boldsymbol{\nu}}•}{|W|•}\,
\int_T\,\mathrm{d}^HV_T(t)\,\int_{G/T}\,\mathrm{d}^HV_{G/T}(g\,T)\nonumber\\
&&\left[\overline{\chi_{k\,\boldsymbol{\nu}}(t)}\,|\Delta(t)|^2\,\varrho'\left( g\,t\,g^{-1},x,y \right)      \,\Pi\left( \tilde{\mu}_{g\,t^{-1}g^{-1}}(x),y  \right)\right]\nonumber\\
&=&\frac{d_{k\,\boldsymbol{\nu}}•}{|W|•}\,\sum_{\sigma\in W}\epsilon(\sigma)\,
\int_T\,\mathrm{d}^H V_T(t)\,\int_{G/T}\,\mathrm{d}^H V_{G/T}(g\,T)\nonumber\\
&&\left[\overline{E_{k\,\boldsymbol{\nu}}(t^{\sigma})}\,\Delta(t)\,\varrho'\left( g\,t\,g^{-1},x,y \right)      \,\Pi\left( \tilde{\mu}_{g\,t^{-1}g^{-1}}(x),y  \right)\right]\nonumber\\
&=&d_{k\,\boldsymbol{\nu}}\,
\int_T\,\mathrm{d}^H V_T(t)\,\int_{G/T}\,\mathrm{d}^H V_{G/T}(g\,T)\nonumber\\
&&\left[\overline{E_{k\,\boldsymbol{\nu}}(t)}\,\Delta(t)\,\varrho'\left( g\,t\,g^{-1},x,y \right)      \,\Pi\left( \tilde{\mu}_{g\,t^{-1}g^{-1}}(x),y  \right)\right].\nonumber
\end{eqnarray}

Choosing a basis of the lattice $L(G)$, we obtain an isomorphism 
$B:(S^1)^{r_G}\cong \mathfrak{t}/L(G)\cong T$;
we shall write the general element of $(S^1)^{r_G}$ as
$e^{\imath\,\boldsymbol{\vartheta}}=\begin{pmatrix}
e^{\imath\,\vartheta_1}&\cdots&e^{\imath\,\vartheta_{r_G}}
\end{pmatrix}$. 
Then with $\mathrm{d}\boldsymbol{\vartheta}=\mathrm{d}\vartheta_1\,\cdots\,\mathrm{d}\vartheta_{r_G}$
$$
B^*\big( \mathrm{d}V_T \big)=\frac{1}{(2\,\pi)^{r_G}}\,\mathrm{d}\boldsymbol{\vartheta},
\quad E_{k\,\boldsymbol{\nu}}\circ B\left(e^{\imath\,\boldsymbol{\vartheta}}\right)=
e^{\imath\,k\,\langle \boldsymbol{\nu},\boldsymbol{\vartheta}  \rangle}
$$
To simplify notation, we shall simply identify $T$ and $(S^1)^{r_G}$,
and write $e^{\imath\,\boldsymbol{\vartheta}}$ for the corresponding element of $T$;
with the same abuse, we shall think of $\imath\,\boldsymbol{\vartheta}$
as an element of $\mathfrak{t}\leqslant \mathfrak{g}$.
Inserting (\ref{eqn:szego oif}) in (\ref{eqn:szego weyl}), we obtain
\begin{eqnarray}
\label{eqn:szego weyl 1}
\Pi_{k\,\boldsymbol{\nu}}(x,y)&\sim&
\frac{•d_{k\,\boldsymbol{\nu}}}{(2\,\pi)^{r_G}•}  \,
\int_{(-\pi,\pi)^{r_G}}\,\mathrm{d}\boldsymbol{\vartheta}\,\int_{G/T}\,\mathrm{d}^HV_{G/T}(g\,T)\\
&&\left[e^{\imath\,\big[u\,\psi(\tilde{\mu}_{g\,e^{-\imath\,\boldsymbol{\vartheta}}g^{-1}}(x),y)-k\,\langle \boldsymbol{\nu},\boldsymbol{\vartheta}  \rangle\big]}\,
\mathcal{A}_k(gT,\boldsymbol{\vartheta},x,y,u)\right]\nonumber\\
&=&\frac{•k\,d_{k\,\boldsymbol{\nu}}}{(2\,\pi)^{r_G}•}  \,
\int_{(-\pi,\pi)^{r_G}}\,\mathrm{d}\boldsymbol{\vartheta}\,
\int_{G/T}\,\mathrm{d}^HV_{G/T}(g\,T)\nonumber\\
&&\left[e^{\imath\,k\,\Gamma (g\,T,\boldsymbol{\vartheta},u,x,y)}\,
\mathcal{A}_k(gT,\boldsymbol{\vartheta},x,y,k\,u)\right]\nonumber
\end{eqnarray}
where
$$
\mathcal{A}_k(gT,\boldsymbol{\vartheta},x,y,u):=\Delta\left(e^{\imath\,\boldsymbol{\vartheta}}\right)\,
\varrho'\left( g\,e^{\imath\,\boldsymbol{\vartheta}}\,g^{-1},x,y \right)\, s\left( \tilde{\mu}_{g\,e^{-\imath\,\boldsymbol{\vartheta}}g^{-1}}(x),y,u  \right),
$$
$$
\Gamma(g\,T,\boldsymbol{\vartheta},u,x,y):= u\,\psi(\tilde{\mu}_{g\,e^{-\imath\,\boldsymbol{\vartheta}}g^{-1}}(x),y)-\langle \boldsymbol{\nu},\boldsymbol{\vartheta}  \rangle .
$$
%
On the support of $\varrho'$ we have 
$\tilde{\mu}_{g\,e^{-\imath\,\boldsymbol{\vartheta}}g^{-1}}(x)\sim y$.
In view of (\ref{eqn:diagodiffpsi}), in any given coordinate system
therefore
\begin{equation}
\label{eqn:approx diagonal dpsi}
\mathrm{d}_{\big( \tilde{\mu}_{g\,e^{-\imath\,\boldsymbol{\vartheta}}g^{-1}}(x), y \big)}\psi
\sim (\alpha_y,-\alpha_y).
\end{equation}
On the other hand, 
by Lemma 2.10 of \cite{pao-IJM} 
in Heisenberg local coordinates centered at $x$ we have
\begin{eqnarray}
\label{eqn:heisenberg local coordinates}
\lefteqn{
\tilde{\mu}_{g\,e^{-\imath\,\boldsymbol{\vartheta}}g^{-1}}(x)=
\tilde{\mu}_{e^{-\mathrm{Ad}_g(\imath\,\boldsymbol{\vartheta})}}(x)
}
\\
&=&x+\Big(\big\langle \Phi(m_x),\mathrm{Ad}_g(\imath\,\boldsymbol{\vartheta})\big\rangle 
+R_3(\boldsymbol{\vartheta}),-\mathrm{Ad}_g(\imath\,\boldsymbol{\vartheta})_M(m_x)
+R_2(\boldsymbol{\vartheta})   \Big)\nonumber\\
&=&x+\Big(\big\langle \Phi\big(\mu_{g^{-1}}(m_x)\big),\imath\,\boldsymbol{\vartheta}\big\rangle 
+R_3(\boldsymbol{\vartheta}),-\mathrm{Ad}_g(\imath\,\boldsymbol{\vartheta})_M(m_x)
+R_2(\boldsymbol{\vartheta})   \Big).\nonumber
\end{eqnarray}
It follows that on the support of $\varrho'$
\begin{equation}
\label{eqn:theta derivative}
\partial_{\boldsymbol{\vartheta}}\Gamma(g\,T,\boldsymbol{\vartheta},u,x,y)
\sim u\, \left.\mathrm{Ad}_{g^{-1}}\Phi(m_y)\right|_{\boldsymbol{\mathfrak{t}}}-\boldsymbol{\nu}.
\end{equation}

Since $\mathcal{O}_{\boldsymbol{\nu}}\cap \mathfrak{t}^0=\emptyset$,
there exists $r_0>0$ such that, with $\Phi^T(m)=\left.\Phi (m)\right|_{\mathfrak{t}}$,
$$
\|\left.\Phi^T (m)\right|_{\mathfrak{t}}\|\ge r_0,\quad \forall m\in M_{\mathcal{O}}.
$$
Hence, if $U_{\mathcal{O}}$ is a sufficiently small open neighborhood of
$M_{\mathcal{O}}$, then
$$
\left\|\Phi^T(m)\right\|\ge \frac{1}{2}\, r_0,\quad \forall m\in U_{\mathcal{O}}.
$$
This applies to 
$\mathrm{Ad}_{g^{-1}}\Phi(m_y)=\Phi\circ \mu_{g^{-1}}(m_y)$. It then follows from
(\ref{eqn:theta derivative}) (arguing as in the proof of Lemma 5.3 of
\cite{gp1}) that the following holds:

\begin{lem}
\label{lem:cmpt supp u}
Suppose $D\gg 0$, and let $\rho\in \mathcal{C}^\infty_c\big( (1/(2\,D), 2\,D) \big)$
be such a bump function that $\rho\equiv 1$ on $(1/D,D)$. Then only a rapidly decreasing 
contribution to the asymptotics of 
(\ref{eqn:szego weyl 1}) is lost, if the integrand is multiplied by 
$\rho (t)$.
\end{lem}

Hence we may assume that integration in $u$ is compactly supported.
The proof is the completed by iteratively integrating by parts in
$\mathrm{d}u$, as in the proof of Proposition 5.2 of \cite{gp1}.

\end{proof}

\section{Proof of Theorem \ref{thm:rapid decrease off locus resc}}

In the proof of Theorem \ref{thm:rapid decrease off locus resc},
we shall rely on the Kirillov character formula (\cite{ki}, \cite{r}), which we briefly recall. 

Let 
$
\exp_G:\boldsymbol{\xi}\in \mathfrak{g}\mapsto e^{\boldsymbol{\xi}}\in G
$ be the exponential map of $G$, and let $\mathfrak{g}'\subseteq \mathfrak{g}$
and $G'\subseteq G$ be open neighborhoods of the origin $\mathbf{0}\in \mathfrak{g}$
and of the unit $1_G\in G$, respectively, such that $\exp_G$ restricts to a
diffeomorphism $\mathfrak{g}'\rightarrow G'$.
We may find an Ad-invariant Euclidean product 
$\varphi^H$ on $\mathfrak{g}$, 
such that $\mathrm{d}^HV_G(g)$ is the Riemannian density
associated to the induced bi-invariant 
Riemannian metric on $G$, which with abuse of notation
we shall also denote by $\varphi^H$.
Let $\mathrm{d}^H\boldsymbol{\xi}$ be the Lebesgue measure on
$\mathfrak{g}$ induced by $\varphi^H$.
Let the $\mathcal{C}^\infty$ function 
$\mathcal{P}:\mathfrak{g}'\rightarrow (0,+\infty)$ be defined by the equality
$$
\exp_G^*( \mathrm{d}^HV_G)=\mathcal{P}^2\,\mathrm{d}^H\boldsymbol{\xi}.
$$
Clearly $\mathcal{P}(0)=1$.

Let us set $n_{G}:=(d_{G}-r_{G})/2$.
Then for every $\boldsymbol{\nu}\in \mathcal{E}^G$ 
and $\boldsymbol{\xi}\in \mathfrak{g}'$ we have
$$
\dim(\mathcal{O}_{\boldsymbol{\nu}})=\dim(G/T)=d_{G}-r_{G}=2\,n_{G}.
$$
Furthermore, let us denote by 
$\sigma_{\boldsymbol{\nu}}$ the Kostant-Kirillov symplectic structure on
$\mathcal{O}_{\boldsymbol{\nu}}$, so that $\sigma_{\boldsymbol{\nu}}^{n_{G}}/n_{G}!$
is the symplectic volume form on $\mathcal{O}_{\boldsymbol{\nu}}$.
In the following we shall set 
$$\mathrm{d}V_{\mathcal{O}_{\boldsymbol{\nu}}}:=\frac{•\sigma_{\boldsymbol{\nu}}^{n_{G}}}{•n_{G}!}, 
\quad \mathrm{vol}\big(\mathcal{O}_{\boldsymbol{\nu}}\big):=
\int_{\mathcal{O}_{\boldsymbol{\nu}}}\,\mathrm{d}V_{\mathcal{O}_{\boldsymbol{\nu}}}.
$$

The Kirillov character formula then says that

\begin{equation}
\label{eqn:kirillovcf}
\chi_{\boldsymbol{\nu}}\left(e^{\boldsymbol{\xi}}\right)=
\frac{1}{(2\,\pi)^{n_{\boldsymbol{\nu}}}}\, \frac{1}{•\mathcal{P} (\boldsymbol{\xi})}\,
\int_{\mathcal{O}_{\boldsymbol{\nu}}}\,e^{\imath\,\langle \boldsymbol{\lambda},
\boldsymbol{\xi}\rangle}\,\mathrm{d}V_{\mathcal{O}_{\boldsymbol{\nu}}}(\boldsymbol{\lambda})\quad (\boldsymbol{\xi}\in \mathfrak{g}').
\end{equation}
Given (\ref{eqn:kirillovcf}) and (\ref{eqn:weyldf}), setting $\boldsymbol{\xi}=\mathbf{0}$ we get
\begin{equation}
\label{eqn:dnuO}
d_{\boldsymbol{\nu}}=\frac{\mathrm{vol}\big(\mathcal{O}_{\boldsymbol{\nu}}\big)}{(2\,\pi)^{n_{G}}•}\,
\quad\Rightarrow\quad
\mathrm{vol}\big(\mathcal{O}_{\boldsymbol{\nu}}\big)=
(2\,\pi)^{n_{G}}\,\prod_{\boldsymbol{\beta}\in R^+}\frac{\varphi (\boldsymbol{\nu},\boldsymbol{\beta})}{\varphi (\boldsymbol{\delta},\boldsymbol{\beta})}.
\end{equation}

For every $k\ge 1$, by a rescaling we obtain
\begin{equation}
\label{eqn:kirillovcfk}
\chi_{k\,\boldsymbol{\nu}}\left(e^{\boldsymbol{\xi}}\right)=
\left(\frac{k}{2\,\pi}\right)^{n_{G}}\, 
\frac{1}{•\mathcal{P} (\boldsymbol{\xi})}\,
\int_{\mathcal{O}_{\boldsymbol{\nu}}}\,e^{\imath\,k\,\langle \boldsymbol{\lambda},
\boldsymbol{\xi}\rangle}\,\frac{•\sigma_{\boldsymbol{\nu}}^{n_{G}}(\boldsymbol{\lambda})}{•n_{G}!}\quad (\boldsymbol{\xi}\in \mathfrak{g}').
\end{equation}
In particular,
\begin{equation}
\label{eqn:kirillovcfkdimension}
d_{k\,\boldsymbol{\nu}}=
\left(\frac{k}{2\,\pi}\right)^{n_{G}}\, 
\mathrm{vol}\big(\mathcal{O}_{\boldsymbol{\nu}}\big).
\end{equation}

\begin{proof}
[Proof of Theorem \ref{thm:rapid decrease off locus resc}]
We may assume without loss that $\epsilon\in (0,1/6)$. 
Furthermore, it suffices to prove the Theorem when $x=y$, 
since by the Cauchy-Schwartz inequality
$$
\left|\Pi^{\tilde{\mu}}_{k\,\boldsymbol{\nu}}(x,y)\right|\le 
\Pi^{\tilde{\mu}}_{k\,\boldsymbol{\nu}}(x,x)^{\frac{1}{2}}\,
\Pi^{\tilde{\mu}}_{k\,\boldsymbol{\nu}}(y,y)^{\frac{1}{2}},
$$
and on the other hand $\Pi^{\tilde{\mu}}_{k\,\boldsymbol{\nu}}(x,x)^{\frac{1}{2}}$
may be seen to satisfy an \textit{a priori} polynomial bound in $k$, by adapting the arguments
in \S 5.1.2 of \cite{gp1}.
By Theorem \ref{thm:rapid decrease}, 
we need only consider the case where $x$ belongs to a small $S^1\times G$-invariant 
neighborhood $V_{\mathcal{O}_{\boldsymbol{\nu}}}\subseteq X$
of $X_{\mathcal{O}_{\boldsymbol{\nu}}}$; 
in particular, we may assume without loss that $\tilde{\mu}$ is free
on $V_{\mathcal{O}_{\boldsymbol{\nu}}}$.
Furthermore, we may replace $x$ by $\tilde{\mu}_g(x)$ for any given $g\in G$,
and assume that
$\Phi (m_x)=\lambda_x\,\boldsymbol{\nu}+\beta_x$, where
$\lambda_x>0$, $\beta_x\in \boldsymbol{\nu}^\perp$, and
$\|\beta_x\|\ll \lambda_x$.

Let us start from (\ref{eqn:szego equiv proj}) with $x=y$:
\begin{equation}
\label{eqn:szego equiv proj diag}
\Pi_{k\,\boldsymbol{\nu}}(x,x)=
d_{k\,\boldsymbol{\nu}}\,\int_G \overline{\chi_{k\,\boldsymbol{\nu}}(g)}\,\Pi\left(
\tilde{\mu}_{g^{-1}}(x),x\right)\,\mathrm{d}^HV_G(g).
\end{equation}

Let us set, for some small $\varepsilon_1>0$,
\begin{eqnarray*}
W'&:=&\left\{ (g,x)\in G\times X\,:\,\mathrm{dist}_X\left(\tilde{\mu}_g(x),x\right)<2\,\varepsilon_1   \right\}\\
W''&:=&\left\{ (g,x)\in G\times X\,:\,\mathrm{dist}_X\left(\tilde{\mu}_g(x),x\right)>\varepsilon_2   \right\}.
\end{eqnarray*}
Let $\rho'+\rho''=1$ be a partition of unity on $G\times X$ subordinate to the open cover
$\{W',W''\}$. Then an argument as in the proof of Theorem \ref{thm:rapid decrease} shows that only a rapidly
decreasing contribution to the asymptotics of (\ref{eqn:szego equiv proj diag}) is lost, if
the integrand is multiplied by $\rho''$. Hence we are reduced to considering the asymptotics
of
\begin{equation}
\label{eqn:szego equiv proj diag1}
\Pi_{k\,\boldsymbol{\nu}}(x,x)':=
d_{k\,\boldsymbol{\nu}}\,\int_G \rho'(g,x)\,\overline{\chi_{k\,\boldsymbol{\nu}}(g)}\,\Pi\left(
\tilde{\mu}_{g^{-1}}(x),x\right)\,\mathrm{d}^HV_G(g).
\end{equation}
Since $\tilde{\mu}$ is free on $V_{\mathcal{O}_{\boldsymbol{\nu}}}$, the partial
function $\rho_x':=\rho'(\cdot,x)$ is supported on a small open neighborhood of the unit
$1_G\in G$, which we may assume to be diffeomorphic to an open neighborhood of 
$\mathbf{0}\in \mathfrak{g}$ by the exponential map.
Hence on the same neighborhood we may set $g=e^{\boldsymbol{\xi}}$
and express $\chi_{k\,\boldsymbol{\nu}}(g)$
by the Kirillov character formula (\ref{eqn:kirillovcfk}). Furthermore, since
$\tilde{\mu}_{g^{-1}}(x)\sim x$ for $\rho'_x(g)\neq 0$, we may also replace $\Pi$
by its description as an FIO (\ref{eqn:szego oif}).
After the rescaling $u\mapsto k\,u$, we obtain
\begin{eqnarray}
\label{eqn:foi+kir}
\lefteqn{\Pi_{k\,\boldsymbol{\nu}}(x,x)\sim \Pi_{k\,\boldsymbol{\nu}}(x,x)'   }\\
&\sim&k\,d_{k\,\boldsymbol{\nu}}\cdot\left(\frac{k}{2\,\pi}\right)^{n_{G}}\, 
\int_0^{+\infty}\,\mathrm{d}u\,
\int_{\mathfrak{g}'}\,\mathrm{d}^H\boldsymbol{\xi}\,\int_{\mathcal{O}_{\boldsymbol{\nu}}}\,
\mathrm{d}V_{\mathcal{O}_{\boldsymbol{\nu}}}(\boldsymbol{\lambda})
\left[
e^{\imath\,k\,\Gamma_x(u,\boldsymbol{\xi},\boldsymbol{\lambda})}\,
\mathcal{A}_{x,k}(u,\boldsymbol{\xi},\boldsymbol{\lambda})
\right]\nonumber
\end{eqnarray}
where
$$
\Gamma_x (u,\boldsymbol{\xi},\boldsymbol{\lambda}):=   u\,\psi\left(
\tilde{\mu}_{e^{-\boldsymbol{\xi}}}(x),x\right)-\langle \boldsymbol{\lambda},
\boldsymbol{\xi}  \rangle
$$
$$
\mathcal{A}_{x,k}(u,\boldsymbol{\xi},\boldsymbol{\lambda}):=
\rho'_x\left(e^{\boldsymbol{\xi}}\right) \,\mathcal{P} (\boldsymbol{\xi})\,s\left(
\tilde{\mu}_{e^{-\boldsymbol{\xi}}}(x),x,k\,u\right).
$$
Since $\tilde{\mu}_{e^{-\boldsymbol{\xi}}}(x)\sim x$ on the support of 
$\mathcal{A}_{x,k}$, we have in local coordinates
$\mathrm{d}_{(\tilde{\mu}_{e^{-\boldsymbol{\xi}}}(x), x)}\psi\sim (\alpha_x,-\alpha_x)$.
Hence
$
\partial_{\boldsymbol{\xi}}\Gamma_x (u,\boldsymbol{\xi},\boldsymbol{\lambda})
\sim u\,\Phi (m_x)-\boldsymbol{\lambda}$.
We then have an analogue of Lemma \ref{lem:cmpt supp u}, so that
integration in $\mathrm{d}u$ may be assumed to be compactly supported.
We express this by multiplying the amplitude in
(\ref{eqn:foi+kir}) by a bump function
$\rho=\rho(u)$ compactly supported in $(1/D,D)$ for some $D\gg 0$. 

Let $\gamma\in \mathcal{C}^\infty(\mathfrak{g})$ be $\ge 0$, supported on a 
ball of radius $2$ centered at the origin (say with respect to $\varphi^H$) and
$\equiv 1$ on a ball of radius $1$ centered at the origin. 
Let us define $\gamma_{k}\in  \mathcal{C}_c^\infty(\mathfrak{g})$ for $k=1,2,\ldots$ by setting
$$\gamma_{k}(\boldsymbol{\xi}):=\gamma \left(k^{1/2-\epsilon}\,\boldsymbol{\xi}  \right).$$
Let $\Pi_{k\,\boldsymbol{\nu}}(x,x)_1$ and $\Pi_{k\,\boldsymbol{\nu}}(x,x)_2$ be given
by the second line of (\ref{eqn:foi+kir}) multiplied by, respectively, 
$\gamma_{k}$ and $1-\gamma_{k}$.

\begin{lem}
\label{lem:pikD2}
$\Pi_{k\,\boldsymbol{\nu}}(x,x)_2=O\left(k^{-\infty}   \right)$
as $k\rightarrow +\infty$.
\end{lem}
\begin{proof}
[Proof of Lemma \ref{lem:pikD2}]
On the support of 
$1-\gamma_{k}$, we have
$\|\boldsymbol{\xi}\|^H\ge k^{\epsilon-\frac{1}{•2}}$
in $\varphi^H$-norm.
Hence for a certain constant $r_0>0$ depending only on the choice of an
invariant open neighborhood $V_{\mathcal{O}}\subseteq X$ of $X_{\mathcal{O}}$ we have
\begin{equation}
\label{eqn:distance xi estimate}
\mathrm{dist}_X\left( \tilde{\mu}_{e^{-\boldsymbol{\xi}}}(x),x  \right)\ge 
r_0\,k^{\epsilon-\frac{1}{•2}}\qquad 
\big(x\in V_{\mathcal{O}},\, \boldsymbol{\xi}\in \mathrm{supp}(1-\gamma_{D_1,k})\big).
\end{equation}
Hence by (\ref{eqn:imaginary part})
\begin{equation}
\label{eqn:u deriv}
\big|\partial_u\Gamma_x (u,\boldsymbol{\xi},\boldsymbol{\lambda})  \big| = \big|\psi\left(
\tilde{\mu}_{e^{-\boldsymbol{\xi}}}(x),x\right)\big|\ge 
\big| \Im \psi\left(
\tilde{\mu}_{e^{-\boldsymbol{\xi}}}(x),x\right) \big|\ge 
D_X\,r_0^2\,k^{2\epsilon-1}.
\end{equation}
Iteratively integrating by parts in $\mathrm{d}u$ then implies the statement, since at each step we
introduce a factor $O\left( k^{-2\,\epsilon} \right)$ (see Proposition 5.2 of \cite{gp1}).
\end{proof}

Thus we are reduced to considering the asymptotics of $\Pi_{k\,\boldsymbol{\nu}}(x,x)_1$.
On the support of 
$\gamma_{k}$ we have 
$\|\boldsymbol{\xi}\|^H\le 2\,k^{\epsilon-\frac{1}{•2}}$.
Let us operate the rescaling 
$\boldsymbol{\xi}\mapsto \boldsymbol{\xi}/\sqrt{k}$, 
and rewrite $\Pi_{k\,\boldsymbol{\nu}}(x,x)_1$ as
\begin{eqnarray}
\label{eqn:foi+kir1}
\Pi_{k\,\boldsymbol{\nu}}(x,x)_1
&=&k^ {1-d_G/2}\,d_{k\,\boldsymbol{\nu}}\cdot\left(\frac{k}{2\,\pi}\right)^{n_{G}}\, 
\int_{1/D}^{D}\,\mathrm{d}u\,
\int_{\mathfrak{g}'}\,\mathrm{d}^H\boldsymbol{\xi}\,\int_{\mathcal{O}_{\boldsymbol{\nu}}}\,
\mathrm{d}V_{\mathcal{O}_{\boldsymbol{\nu}}}(\boldsymbol{\lambda})\nonumber\\
&&\left[
e^{\imath\,k\,\Gamma_x(u,\boldsymbol{\xi}/\sqrt{k},\boldsymbol{\lambda})}\,
\mathcal{A}_{x,k}(u,\boldsymbol{\xi}/\sqrt{k},\boldsymbol{\lambda})\,\rho (u)\,
\gamma \left( k^{-\epsilon}\,\boldsymbol{\xi}  \right)
\right].
\end{eqnarray}
Now integration in $\boldsymbol{\xi}$ is on an expanding ball of radius 
$O\left(k^{\epsilon}\right)$ centered at the origin.
In Heisenberg local coordinates at $x$, 
by Lemma 2.10 of \cite{pao-IJM} we have
\begin{eqnarray}
\label{eqn:heisenberg local coordinates1}
\lefteqn{
\tilde{\mu}_{e^{-\boldsymbol{\xi}/\sqrt{k}}}(x)
}\\
&=&x+\left(\frac{1}{\sqrt{k}}\,\big\langle \Phi(m_x),\boldsymbol{\xi})\big\rangle 
+R_3\left(\frac{1}{\sqrt{k}}\,\boldsymbol{\xi}\right),-\frac{1}{\sqrt{k}}\,\boldsymbol{\xi}_M(m_x)+R_2\left(\frac{1}{\sqrt{k}}\,\boldsymbol{\xi}\right)   \right).\nonumber
\end{eqnarray}
As in the proof of Theorem 1 of \cite{pao-IJM},
using the expansions in \S 3 of \cite{sz} one gets
\begin{eqnarray}
\label{eqn:heisenberg local coordinates2}
\lefteqn{
\Gamma_x \left(u,\boldsymbol{\xi}/\sqrt{k},\boldsymbol{\lambda}\right)
}\\
&=&\frac{1}{\sqrt{k}}\,\langle u\,\Phi (m_x)-\boldsymbol{\lambda},\boldsymbol{\xi}   \rangle 
+\frac{\imath\,u}{2\,k}\,\|\boldsymbol{\xi}_X(x)\|^2
+u\,R_3\left(   \frac{•\boldsymbol{\xi}}{\sqrt{k}}\right)\,e^{\imath\,\varsigma_{x,k}(\boldsymbol{\xi}
/\sqrt{k})},\nonumber
\end{eqnarray}
where 
$\varsigma_{x,k}(\boldsymbol{\xi})=
\big\langle \Phi(m_x),\boldsymbol{\xi})\big\rangle +R_3(\boldsymbol{\xi})$.
Hence we rewrite (\ref{eqn:foi+kir1}) as follows
\begin{eqnarray}
\label{eqn:foi+kir2}
\Pi_{k\,\boldsymbol{\nu}}(x,x)_1
&=&k^ {1-d_G/2}\,d_{k\,\boldsymbol{\nu}}\cdot\left(\frac{k}{2\,\pi}\right)^{n_{G}}\, 
\int_{1/D}^{D}\,\mathrm{d}u\,
\int_{\mathfrak{g}'}\,\mathrm{d}^H\boldsymbol{\xi}\,\int_{\mathcal{O}_{\boldsymbol{\nu}}}\,
\mathrm{d}V_{\mathcal{O}_{\boldsymbol{\nu}}}(\boldsymbol{\lambda})\nonumber\\
&&\left[
e^{\imath\,\sqrt{k}\,\Upsilon_x(u,\boldsymbol{\xi},\boldsymbol{\lambda})}\,
\mathcal{B}_{x,k}(u,\boldsymbol{\xi}/\sqrt{k},\boldsymbol{\lambda})\,
\gamma \left(D_1\, k^{-\epsilon}\,\boldsymbol{\xi}  \right)
\right],
\end{eqnarray}
where now
\begin{eqnarray*}
\Upsilon_x(u,\boldsymbol{\xi},\boldsymbol{\lambda})&:=&\langle u\,\Phi (m_x)-\boldsymbol{\lambda},\boldsymbol{\xi}   \rangle\\
\mathcal{B}_{x,k}(u,\boldsymbol{\xi}/\sqrt{k},\boldsymbol{\lambda})&:=&
e^{-\frac{u}{2}\,\|\boldsymbol{\xi}_X(x)\|^2}\cdot \rho (u)\,
\mathcal{A}_{x,k}(u,k^{-1/2}\,\boldsymbol{\xi},\boldsymbol{\lambda})\,
e^{u\,k\,R_3\left(   \frac{•\boldsymbol{\xi}}{\sqrt{k}}\right)\,e^{\imath\,\varsigma_{x,k}(\boldsymbol{\xi}
/\sqrt{k})}}
\end{eqnarray*}
We have $\boldsymbol{\xi}_X(x)\neq 0$ if $x\in V_{\mathcal{O}_{\boldsymbol{\nu}}}$,
$\boldsymbol{\xi}\neq 0$.

Under the present transversality assumption (Assumption \ref{asmpt:basic1}),
there exists $s_0>0$ such that 
$$
\big\| \Phi(m)-\boldsymbol{\lambda}    \big\|\ge s_0\cdot \mathrm{dist}_M(m,M_{\mathcal{O}}),
\qquad \,\forall\,m\in M,\, \forall \boldsymbol{\lambda}\in \mathcal{C}(\mathcal{O}_{\boldsymbol{\nu}}).
$$
Therefore, in the situation of the Theorem,
\begin{equation*}
\|\partial_{\boldsymbol{\xi}}\Upsilon_x(u,\boldsymbol{\xi},\boldsymbol{\lambda})\big\|
=\big\|u\,\Phi (m_x)-\boldsymbol{\lambda}\big\|
\ge \frac{•s_0}{•D} \cdot C\,k^{\epsilon-\frac{1}{•2}},\quad \forall\,
\boldsymbol{\lambda}\in \mathcal{O}_{\boldsymbol{\nu}},\quad 
\forall\,u\in \left( \frac{1}{D},D  \right).
\end{equation*}
The statement of the Theorem then follows by iteratively integrating 
by parts in $\mathrm{d}\boldsymbol{\xi}$, since each step introduces a
factor $O\left(k^{-\epsilon}\right)$.
\end{proof}

\section{Proof of Theorem \ref{thm:scal asymp}}

Before delving into the proof of Theorem \ref{thm:scal asymp}, let us make the following
remarks. 

Suppose $x\in X$, and let $\varpi_x:(\theta,\mathbf{v})\mapsto x+(\theta,\mathbf{v})$ be
a system of Heisenberg local coordinates at $x$. Then 
$\varpi_x$ induces an isomorphism $T_xX\cong \mathbb{R}\times \mathbb{R}^{2n}$, 
in terms of which we can give a meaning to the expression
$x+\upsilon$, when $\upsilon\in T_xX$ is small.
For some $c_1>c_2>0$ we have
\begin{equation}
\label{eqn:distance estimate}
c_2\,\|\upsilon_1-\upsilon_2\|\le \mathrm{dist}_X(x+\upsilon_1,x+\upsilon_2)
\le c_1\,\|\upsilon_1-\upsilon_2\|
\end{equation}
if $\upsilon_j\sim \mathbf{0}$.

Let $\mathfrak{t}_m'$ be as in (\ref{eqn:t_m' characterized});
we have the following characterization of the normal 
bundle $N\big(M_{\mathcal{O}_{\boldsymbol{\nu}}}\big)$ 
to $M_{\mathcal{O}_{\boldsymbol{\nu}}}$ in $M$, 
which can be proved by minor adaptations of the arguments used in 
Lemma 4.2 and Step 4.3 of \cite{gp1}.

\begin{lem}
\label{lem:normal bundle}
For any $m\in M_{\mathcal{O}_{\boldsymbol{\nu}}}$,
$N_m\big(M_{\mathcal{O}_{\boldsymbol{\nu}}}\big)=J_m\big( \mathfrak{t}_m'  \big)$.

\end{lem}

Furthermore, we may identify the normal bundle of 
$X_{\mathcal{O}_{\boldsymbol{\nu}}}\subseteq X$, 
$N\big(X_{\mathcal{O}_{\boldsymbol{\nu}}}\big)$, 
with the pull-back of
$N\big(M_{\mathcal{O}_{\boldsymbol{\nu}}}\big)$; even more explicitly,
for every $x\in X_{\mathcal{O}_{\boldsymbol{\nu}}}$ we have with $m_x=\pi(x)$
$$
N_x\big(X_{\mathcal{O}_{\boldsymbol{\nu}}}\big)=
N_{m_x}\big(X_{\mathcal{O}_{\boldsymbol{\nu}}}\big)^\sharp.
$$
Hence there is an orthogonal direct sum
\begin{equation}
\label{eqn:direct sum orbit}
N_{m_x}\big(X_{\mathcal{O}_{\boldsymbol{\nu}}}\big)^\sharp
\oplus \mathfrak{g}_X(x)\oplus
{\mathfrak{g}_M(m_x)^{\perp_h}}^\sharp\subseteq T_xX.
\end{equation}

We then have the following consequence, whose proof is omitted.

\begin{lem}
\label{lem:positive distance resc}
Suppose $x\in X_{\mathcal{O}_{\boldsymbol{\nu}}}$, and choose a system
of Heisenberg local coordinates at $x$.
Then there exists $\delta>0$ such that
for any choice of
$\boldsymbol{\xi}\in \mathfrak{g}$,
$\mathbf{v}_j\in N_{m_x}\big(X_{\mathcal{O}_{\boldsymbol{\nu}}}\big)$,
$\mathbf{w}_j\in {\mathfrak{g}_M(m_x)^{\perp_h}}$ of sufficiently small norm
we have
$$
\mathrm{dist}_X\left(\tilde{\mu}_{e^{-\boldsymbol{\xi}}}\big(x+(\mathbf{v}_1+\mathbf{w}_1)\big),
x+(\mathbf{v}_2+\mathbf{w}_2)   \right)\ge 
\delta\,\|\boldsymbol{\xi}\|^\varphi.
$$
Furthermore, $\delta$ may be chosen uniformly on $X_{\mathcal{O}_{\boldsymbol{\nu}}}$.
\end{lem}

\begin{proof}
[Proof of Theorem \ref{thm:scal asymp}]
We may replace $x$ by $\tilde{\mu}_h(x)$ for a suitable $h\in G$,
and assume without loss that
$\Phi(m_x)=\varsigma (m_x)\,\boldsymbol{\nu}$.
Hence, we may assume that
\begin{equation}
\label{eqn:reductiontmtm'}
\mathfrak{t}_m=\mathfrak{t},\quad \mathfrak{t}_m'=\mathfrak{t}_{\boldsymbol{\nu}}
\end{equation}
(see (\ref{eqn:t_m characterized}), (\ref{eqn:t_m' characterized})).
Let us set
$$x_{j,k}:=\frac{1}{\sqrt{k}}\,(\mathbf{v}_j+\mathbf{w}_j)   
\qquad (j=1,2;\, k>0).$$
and replace (\ref{eqn:szego equiv proj diag}) by
\begin{equation}
\label{eqn:szego equiv proj near diag}
\Pi_{k\,\boldsymbol{\nu}}(x_{1,k},x_{2,k})=
d_{k\,\boldsymbol{\nu}}\,\int_G \overline{\chi_{k\,\boldsymbol{\nu}}(g)}\,\Pi\left(
\tilde{\mu}_{g^{-1}}(x_{1,k}),x_{2,k}\right)\,\mathrm{d}^HV_G(g).
\end{equation}
Given that $x_{j,k}\rightarrow x$, the argument leading to
(\ref{eqn:szego equiv proj diag1}) now implies
$\Pi_{k\,\boldsymbol{\nu}}(x_{1,k},x_{2,k})\sim \Pi_{k\,\boldsymbol{\nu}}(x_{1,k},x_{2,k})'$,
where
\begin{equation}
\label{eqn:szego equiv proj near diag1}
\Pi_{k\,\boldsymbol{\nu}}(x_{1,k},x_{2,k})':=
d_{k\,\boldsymbol{\nu}}\,\int_G \rho'(g,x)\,\overline{\chi_{k\,\boldsymbol{\nu}}(g)}\,\Pi\left(
\tilde{\mu}_{g^{-1}}(x_{1,k}),x_{2,k}\right)\,\mathrm{d}^HV_G(g).
\end{equation}
We then obtain in place of (\ref{eqn:foi+kir})
\begin{eqnarray}
\label{eqn:foi+kirk}
\lefteqn{\Pi_{k\,\boldsymbol{\nu}}(x_{1,k},x_{2,k})\sim 
\Pi_{k\,\boldsymbol{\nu}}(x_{1,k},x_{2,k})'   }\\
&\sim&k\,d_{k\,\boldsymbol{\nu}}\cdot\left(\frac{k}{2\,\pi}\right)^{n_{G}}\,\, 
\int_0^{+\infty}\,\mathrm{d}u\,
\int_{\mathfrak{g}'}\,\mathrm{d}^H\boldsymbol{\xi}\,\int_{\mathcal{O}_{\boldsymbol{\nu}}}\,
\mathrm{d}V_{\mathcal{O}_{\boldsymbol{\nu}}}(\boldsymbol{\lambda})
\left[
e^{\imath\,k\,\Gamma_{x,k}(u,\boldsymbol{\xi},\boldsymbol{\lambda})}\,
\mathcal{B}_{x,k}(u,\boldsymbol{\xi},\boldsymbol{\lambda})
\right]\nonumber
\end{eqnarray}
where
\begin{equation}
\label{eqn:gamma xk}
\Gamma_{x,k} (u,\boldsymbol{\xi},\boldsymbol{\lambda}):=   u\,\psi\left(
\tilde{\mu}_{e^{-\boldsymbol{\xi}}}(x_{1,k}),x_{2,k}\right)-\langle \boldsymbol{\lambda},
\boldsymbol{\xi}  \rangle
\end{equation}
\begin{equation}
\label{eqn:B xk}
\mathcal{B}_{x,k}(u,\boldsymbol{\xi},\boldsymbol{\lambda}):=
\rho'_x\left(e^{\boldsymbol{\xi}}\right) \,\mathcal{P} (\boldsymbol{\xi})\cdot 
s\left(\tilde{\mu}_{e^{-\boldsymbol{\xi}}}(x_{1,k}),x_{2,k},k\,u\right).
\end{equation}

Since $\tilde{\mu}_{e^{-\boldsymbol{\xi}}}(x_{1,k})\sim x_{2,k}$ on the support of 
$\mathcal{B}_{x,k}$, by the same argument used in the proof 
of Theorem \ref{thm:rapid decrease off locus resc} we may multiply the integrand
in (\ref{eqn:foi+kirk}) by the same cut-off function 
$\rho=\rho(u)$ without affecting the asymptotics, so as to assume
that integration in $\mathrm{d}u$ is supported in $(1/D,D)$ for some $D\gg 0$.

In view of Lemma \ref{lem:positive distance resc}, we have for $k\gg 0$
$$
\mathrm{dist}_X\left(\tilde{\mu}_{e^{-\boldsymbol{\xi}}}\big(x_{1,k})\big),
x_{2,k})   \right)\ge 
\delta\,\|\boldsymbol{\xi}\|^\varphi.
$$
Using this, we obtain an obvious analogue of (\ref{eqn:distance xi estimate}),
so that
we can reprove Lemma \ref{lem:pikD2} in the present setting.
Rescaling in $\boldsymbol{\xi}$, we obtain in place of (\ref{eqn:foi+kir1}):
\begin{eqnarray}
\label{eqn:foi+kir1k}
\lefteqn{
\Pi_{k\,\boldsymbol{\nu}}(x_{1,k},x_{2,k})
}\\
&\sim&k^ {1-d_G/2}\,d_{k\,\boldsymbol{\nu}}\cdot\left(\frac{k}{2\,\pi}\right)^{n_{G}}\,  
\int_{1/D}^{D}\,\mathrm{d}u\,
\int_{\mathfrak{g}}\,\mathrm{d}^H\boldsymbol{\xi}\,\int_{\mathcal{O}_{\boldsymbol{\nu}}}\,
\mathrm{d}V_{\mathcal{O}_{\boldsymbol{\nu}}}(\boldsymbol{\lambda})\nonumber\\
&&\left[
e^{\imath\,k\,\Gamma_{x,k}(u,\boldsymbol{\xi}/\sqrt{k},\boldsymbol{\lambda})}\,
\mathcal{B}_{x,k}(u,\boldsymbol{\xi}/\sqrt{k},\boldsymbol{\lambda})\,\rho (u)\,
\gamma \left( k^{-\epsilon}\,\boldsymbol{\xi}  \right)
\right].\nonumber
\end{eqnarray}

In view of Corollary 2.2 of \cite{pao-IJM},
and using that
$\omega_{m_x}\big( \boldsymbol{\xi}_M(m), \mathbf{w}_1 \big)=0$,
in place of (\ref{eqn:heisenberg local coordinates1}) we have
\begin{eqnarray}
\label{eqn:heisenberg local coordinates1k}
\tilde{\mu}_{e^{-\boldsymbol{\xi}/\sqrt{k}}}(x_{1,k})
=x+\left(\Theta_{k,1},V_{k,1}\right).
\end{eqnarray}
where
\begin{eqnarray}
\label{eqn:Theta_k1}
\Theta_{k,1}&=&\Theta_k(x,\boldsymbol{\xi},\mathbf{v}_1,\mathbf{w}_1)\\
&=&
\frac{1}{\sqrt{k}}\,\big\langle \Phi(m_x),\boldsymbol{\xi})\big\rangle 
+\frac{1}{k}\,\omega_{m_x}\big( \boldsymbol{\xi}_M(m_x), \mathbf{v}_1 \big)
+R_3\left(\frac{1}{\sqrt{k}}\,\boldsymbol{\xi},\frac{1}{\sqrt{k}}\,\mathbf{v}_1,
\frac{1}{\sqrt{k}}\,\mathbf{w}_1\right),\nonumber
\end{eqnarray}
\begin{eqnarray}
\label{eqn:V_k1}
V_{k,1}&=&V_k(x,\boldsymbol{\xi},\mathbf{v}_1,\mathbf{w}_1)\\
&=&
\frac{1}{\sqrt{k}}\,\big(\mathbf{v}_1+\mathbf{w}_1-\boldsymbol{\xi}_M(m_x)\big)+R_2\left(\frac{1}{\sqrt{k}}\,\boldsymbol{\xi},\frac{1}{\sqrt{k}}\,\mathbf{v}_1,
\frac{1}{\sqrt{k}}\,\mathbf{w}_1\right).\nonumber
\end{eqnarray}
We then have (see \S 3 of \cite{sz})
\begin{eqnarray}
\label{eqn:expansion psik}
\lefteqn{u\,\psi\left(
\tilde{\mu}_{e^{-\boldsymbol{\xi}/\sqrt{k}}}(x_{1,k}),x_{2,k}\right)}\\
&=&   \imath\,u\,\left[1-e^{\imath\,\Theta_k}\right]-\imath\,u\,\psi_2
\left(V_{k,1},\frac{1}{\sqrt{k}}\,(\mathbf{v}_2+\mathbf{w}_2)    \right)
+u\,R_3\left(\frac{\boldsymbol{\xi}}{\sqrt{k}},\frac{\mathbf{v}_j}{\sqrt{k}},
\frac{\mathbf{w}_j}{\sqrt{k}} \right).\nonumber
\end{eqnarray}
We have
\begin{eqnarray}
\label{eqn:exp 1st thetak}
\imath\,u\,\left[1-e^{\imath\,\Theta_{k,1}}\right]
&=& u\,\Theta_{k,1}+\frac{\imath\,u}{2}\,\Theta_{k,1}^2
+ u\,R_3\left(\frac{\boldsymbol{\xi}}{\sqrt{k}},\frac{\mathbf{v}_1}{\sqrt{k}},
\frac{\mathbf{w}_1}{\sqrt{k}} \right)\\
&=&\frac{u}{\sqrt{k}}\,\big\langle \Phi(m_x),\boldsymbol{\xi})\big\rangle
+\frac{u}{k}\,\omega_{m_x}\big( \boldsymbol{\xi}_M(m_x),\mathbf{v}_1\big)
+\frac{\imath\,u}{2\,k}\,\big\langle \Phi(m_x),\boldsymbol{\xi}  \big\rangle^2\nonumber\\
&&+u\,R_3\left(\frac{\boldsymbol{\xi}}{\sqrt{k}},\frac{\mathbf{v}_1}{\sqrt{k}},
\frac{\mathbf{w}_1}{\sqrt{k}} \right),\nonumber
\end{eqnarray}
and
\begin{eqnarray}
\label{eqn:psi2 vk}
\lefteqn{\psi_2
\left(V_{k,1},\frac{1}{\sqrt{k}}\,(\mathbf{v}_2+\mathbf{w}_2)    \right)}\\
&=&\frac{1}{k}\,\left[ -\imath\,\omega_{m_x}\big(\mathbf{v}_1+\mathbf{w}_1-\boldsymbol{\xi}_M(m_x),
\mathbf{v}_2+\mathbf{w}_2    \big)  -\frac{1}{2}\,\big\|\big (\mathbf{v}_1-\mathbf{v}_2)+
(\mathbf{w}_1-\mathbf{w}_2)-\boldsymbol{\xi}_M(m_x)  \big\|_{m_x}^2\right]\nonumber\\
&&+R_3\left(\frac{\boldsymbol{\xi}}{\sqrt{k}},\frac{\mathbf{v}_j}{\sqrt{k}},
\frac{\mathbf{w}_j}{\sqrt{k}} \right).\nonumber
\end{eqnarray}
Since $\mathbf{w}_j\in \mathfrak{g}_M(m)^{\perp_{h_m}}$, we have
$\omega_m\big( \boldsymbol{\xi}_M(m_x),\mathbf{w}_2   \big)=0$ 
for any $\boldsymbol{\xi}\in \mathfrak{g}$.

\begin{lem}
\label{lem:v_j involution}
If $m\in M_{\mathcal{O}_{\boldsymbol{\nu}}}$ and 
$\mathbf{v}_j\in N_{m}(M_{\mathcal{O}_{\boldsymbol{\nu}}})$,
then $\omega_m(\mathbf{v}_1,\mathbf{v}_2)=0$.

\end{lem}

\begin{proof}
[Proof of Lemma \ref{lem:v_j involution}]
By Lemma \ref{lem:normal bundle}, there are $\boldsymbol{\eta}_j\in \mathfrak{t}'_m
\subseteq \mathfrak{t}_m$ such that 
$\mathbf{v}_j=J_m\big( {\boldsymbol{\eta}_j}_M(m)  \big)$ (given our previous reduction
we may assume $\mathfrak{t}_m=\mathfrak{t}$).
Hence 
$\omega_m(\mathbf{v}_1,\mathbf{v}_2)=
\omega_m \big ( {\boldsymbol{\eta}_1}_M(m), {\boldsymbol{\eta}_2}_M(m)\big)$.
On the other hand, $\mu$ restricts to a Hamiltonian action of the maximal torus
$T_m$, and therefore the vector fields $\boldsymbol{\eta}_M$, with $\boldsymbol{\eta}\in \mathfrak{t}$,
are all in symplectic involution. Hence 
$\omega \big ( {\boldsymbol{\eta}_1}_M, {\boldsymbol{\eta}_2}_M\big)\equiv 0$.
Hence $\omega_m(\mathbf{v}_1,\mathbf{v}_2)=0$.
\end{proof}

Given (\ref{eqn:direct sum orbit}) and Lemma \ref{lem:v_j involution},
in view of Definition \ref{defn:psi2}
we may rewrite (\ref{eqn:psi2 vk}) as
\begin{eqnarray}
\label{eqn:psi2 vkk}
\lefteqn{\psi_2
\left(V_{k,1},\frac{1}{\sqrt{k}}\,(\mathbf{v}_2+\mathbf{w}_2)    \right)}\\
&=&\frac{1}{k}\,\left[ 
 \psi_2(\mathbf{w}_1,\mathbf{w}_2)  -\frac{1}{2}\,\big\|\mathbf{v}_1-\mathbf{v}_2\|_{m_x}^2 
 +\imath\,\omega_{m_x}\big(\boldsymbol{\xi}_M(m_x),
\mathbf{v}_2\big)
-\frac{1}{2}\,\|\boldsymbol{\xi}_M(m_x)  \big\|_m^2\right]\nonumber\\
&&+R_3\left(\frac{\boldsymbol{\xi}}{\sqrt{k}},\frac{\mathbf{v}_j}{\sqrt{k}},
\frac{\mathbf{w}_j}{\sqrt{k}} \right).\nonumber
\end{eqnarray}

Hence, (\ref{eqn:expansion psik}) may be rewritten
\begin{eqnarray}
\label{eqn:expansion psikexp}
\lefteqn{u\,\psi\left(
\tilde{\mu}_{e^{-\boldsymbol{\xi}/\sqrt{k}}}(x_{1,k}),x_{2,k}\right)}\\
&=& \frac{u}{\sqrt{k}}\,\big\langle \Phi(m_x),\boldsymbol{\xi})\big\rangle
+\frac{u}{k}\,\omega_{m_x}\big( \boldsymbol{\xi}_M(m_x),\mathbf{v}_1\big)
+\frac{\imath\,u}{2\,k}\,\big\langle \Phi(m_x),\boldsymbol{\xi}  \big\rangle^2\nonumber\\
&&+u\,R_3\left(\frac{\boldsymbol{\xi}}{\sqrt{k}},\frac{\mathbf{v}_1}{\sqrt{k}},
\frac{\mathbf{w}_1}{\sqrt{k}} \right)\nonumber\\
&&-\frac{\imath\,u}{k}\,\left[ 
 \psi_2(\mathbf{w}_1,\mathbf{w}_2)  -\frac{1}{2}\,\big\|\mathbf{v}_1-\mathbf{v}_2\|_{m_x}^2 
 +\imath\,\omega_{m_x}\big(\boldsymbol{\xi}_M(m_x),
\mathbf{v}_2\big)
-\frac{1}{2}\,\|\boldsymbol{\xi}_M(m_x)  \big\|_{m_x}^2\right]\nonumber\\
&&+u\,R_3\left(\frac{\boldsymbol{\xi}}{\sqrt{k}},\frac{\mathbf{v}_j}{\sqrt{k}},
\frac{\mathbf{w}_j}{\sqrt{k}} \right).\nonumber
\end{eqnarray}
Whence
\begin{eqnarray}
\label{eqn:expansion psikexpik}
\lefteqn{\imath\,k\,u\,\psi\left(
\tilde{\mu}_{e^{-\boldsymbol{\xi}/\sqrt{k}}}(x_{1,k}),x_{2,k}\right)}\\
&=&\imath\,\sqrt{k} \,u\,\big\langle \Phi(m_x),\boldsymbol{\xi}\big\rangle
+\imath\,u\,\omega_{m_x}\big( \boldsymbol{\xi}_M(m_x),\mathbf{v}_1+\mathbf{v}_2\big)
-\frac{u}{2}\,\big\langle \Phi(m_x),\boldsymbol{\xi}  \big\rangle^2\nonumber\\
&&+u\,\left[ 
 \psi_2(\mathbf{w}_1,\mathbf{w}_2)  -\frac{1}{2}\,\big\|\mathbf{v}_1-\mathbf{v}_2\|_{m_x}^2 
-\frac{1}{2}\,\|\boldsymbol{\xi}_M(m_x)  \big\|_{m_x}^2\right]\nonumber\\
&&+\imath\, u\,k\,R_3\left(\frac{\boldsymbol{\xi}}{\sqrt{k}},\frac{\mathbf{v}_j}{\sqrt{k}},
\frac{\mathbf{w}_j}{\sqrt{k}} \right)\nonumber\\
&=&\imath\,\sqrt{k} \,u\,\big\langle \Phi(m_x),\boldsymbol{\xi}\big\rangle
+\imath\,u\,\omega_{m_x}\big( \boldsymbol{\xi}_M(m_x),\mathbf{v}_1+\mathbf{v}_2\big)
-\frac{u}{2}\,\|\boldsymbol{\xi}_X(x)  \big\|_x^2\nonumber\\
&&+u\,\left[ 
 \psi_2(\mathbf{w}_1,\mathbf{w}_2)  -\frac{1}{2}\,\big\|\mathbf{v}_1-\mathbf{v}_2\|_{m_x}^2 
\right]\nonumber\\
&&+\imath\, u\,k\,R_3\left(\frac{\boldsymbol{\xi}}{\sqrt{k}},\frac{\mathbf{v}_j}{\sqrt{k}},
\frac{\mathbf{w}_j}{\sqrt{k}} \right).\nonumber
\end{eqnarray}

In view of (\ref{eqn:gamma xk})
\begin{eqnarray}
\label{eqn:gamma xkk}
\lefteqn{\imath\,k\,\Gamma_{x,k}(u,\boldsymbol{\xi}/\sqrt{k},\boldsymbol{\lambda})}\\
&=& 
\imath\,k \left[  u\,\psi\left(
\tilde{\mu}_{e^{-\boldsymbol{\xi}/\sqrt{k}}}(x_{1,k}),x_{2,k}\right)-\left\langle \boldsymbol{\lambda},
\frac{1}{\sqrt{k}}\,\boldsymbol{\xi}\,  \right\rangle\right]\nonumber\\
&=&\imath\,\sqrt{k} \,\big\langle u\,\Phi(m_x)-\boldsymbol{\lambda},\boldsymbol{\xi}\big\rangle
+\imath\,u\,\omega_m\big( \boldsymbol{\xi}_M(m),\mathbf{v}_1+\mathbf{v}_2\big)
-\frac{u}{2}\,\|\boldsymbol{\xi}_X(x)  \|_x^2\nonumber\\
&&+u\,\left[ 
 \psi_2(\mathbf{w}_1,\mathbf{w}_2)  -\frac{1}{2}\,\|\mathbf{v}_1-\mathbf{v}_2\|_m^2 
\right]\nonumber\\
&&+\imath\, u\,k\,R_3\left(\frac{\boldsymbol{\xi}}{\sqrt{k}},\frac{\mathbf{v}_j}{\sqrt{k}},
\frac{\mathbf{w}_j}{\sqrt{k}} \right)\nonumber
\end{eqnarray}

Let $\mathrm{d}^\varphi\boldsymbol{\xi}$ be the Lebesgue measure on $\mathfrak{g}$
associated to $\varphi$. Then
\begin{equation}
\label{eqn:lebesgue measure compd}
\mathrm{d}^H\boldsymbol{\xi}=\mathrm{vol}^{\varphi}(G)^{-1}\,\mathrm{d}^\varphi \boldsymbol{\xi}.
\end{equation}
Using this and (\ref{eqn:kirillovcfkdimension}), 
(\ref{eqn:foi+kir1k}) may be rewritten as an oscillatory integral in 
$\sqrt{k}$ with a real phase, in the form
\begin{eqnarray}
\label{eqn:foi+kir1kvarphi}
\lefteqn{
\Pi_{k\,\boldsymbol{\nu}}(x_{1,k},x_{2,k})
}\\
&\sim&\frac{\mathrm{vol}\big(\mathcal{O}_{\boldsymbol{\nu}}\big)}{\mathrm{vol}^{\varphi}(G)} \cdot
k^ {1-d_G/2}\,\left(\frac{k}{2\,\pi}\right)^{2\,n_{G}}\,  
\nonumber\\
&&\cdot\int_{1/D}^{D}\,\mathrm{d}u\,
\int_{\mathfrak{g}}\,\mathrm{d}^\varphi\boldsymbol{\xi}\,\int_{\mathcal{O}_{\boldsymbol{\nu}}}\,
\mathrm{d}V_{\mathcal{O}_{\boldsymbol{\nu}}}(\boldsymbol{\lambda})\left[
e^{\imath\,\sqrt{k}\,\Upsilon_x (u,\boldsymbol{\xi},\boldsymbol{\lambda})}\,
\mathcal{C}_{x,k}(u,\boldsymbol{\xi},\boldsymbol{\lambda};\mathbf{v}_j,\mathbf{w}_j)
\right],\nonumber
\end{eqnarray}
with phase $\Upsilon_x$ and amplitude $\mathcal{C}_{x,k}$ given by,
respectively,
\begin{equation}
\label{eqn:Upsilon_x}
\Upsilon_x (u,\boldsymbol{\xi},\boldsymbol{\lambda}):=
\big\langle u\,\Phi(m_x)-\boldsymbol{\lambda},\boldsymbol{\xi}\big\rangle,
\end{equation}
\begin{eqnarray}
\label{eqn:Cxk}
\mathcal{C}_{x,k}(u,\boldsymbol{\xi},\boldsymbol{\lambda};\mathbf{v}_j,\mathbf{w}_j)&:=&
e^{\imath\,u\,\omega_m\big( \boldsymbol{\xi}_M(m),\mathbf{v}_1+\mathbf{v}_2\big)
-\frac{u}{2}\,\|\boldsymbol{\xi}_X(x)  \|_x^2+u\,\left[ 
 \psi_2(\mathbf{w}_1,\mathbf{w}_2)  -\frac{1}{2}\,\|\mathbf{v}_1-\mathbf{v}_2\|_{m_x}^2 
\right]}\nonumber\\
&&\cdot \mathcal{B}_{x,k}'(u,\boldsymbol{\xi}/\sqrt{k},\boldsymbol{\lambda})
\end{eqnarray}
where
\begin{eqnarray}
\label{eqn:B'xk}
\lefteqn{\mathcal{B}_{x,k}'(u,\boldsymbol{\xi}/\sqrt{k},\boldsymbol{\lambda})}\\
&=&
e^{\imath\, u\,k\,R_3\left(\frac{\boldsymbol{\xi}}{\sqrt{k}},\frac{\mathbf{v}_j}{\sqrt{k}},
\frac{\mathbf{w}_j}{\sqrt{k}} \right)}\cdot 
\mathcal{B}_{x,k}(u,\boldsymbol{\xi}/\sqrt{k},\boldsymbol{\lambda})\,\rho (u)\,
\gamma \left( k^{-\epsilon}\,\boldsymbol{\xi}  \right).\nonumber
\end{eqnarray}
There exists $r_{\boldsymbol{\nu}}>0$, depending only on $\boldsymbol{\nu}$, such that 
$\|\boldsymbol{\xi}_X(x)\|_x\ge r_{\boldsymbol{\nu}}\,\|\boldsymbol{\xi}\|^\varphi$,
$\forall\,x\in X_{\mathcal{O}_{\boldsymbol{\nu}}}$.
Furthermore, Taylor expansion at the origin yields an asymptotic expansion 
of the form
\begin{equation}
\label{eqn:Bxk'expanded}
\mathcal{B}_{x,k}'(u,\boldsymbol{\xi}/\sqrt{k},\boldsymbol{\lambda})\sim
\gamma \left( k^{-\epsilon}\,\boldsymbol{\xi}  \right)\cdot 
\sum_{j\ge 0}k^{d-j/2}\,P_j(m_x,u;\boldsymbol{\xi},\mathbf{v}_j,\mathbf{w}_j),
\end{equation}
where $P_j(m_x,u;\cdot,\cdot,\cdot)$ is a polynomial of degree $\le 3\,j$,
and parity $j$ (recall that 
$\|\boldsymbol{\xi}\|^{\varphi},\,\|\mathbf{v}_j\|,\,\|\mathbf{w}_j\|\le C'\,
k^{\epsilon}$ for some fixed $C'>0$, and that $\epsilon\in (0,1/6)$). 
In view of (\ref{eqn:amplitude szego expanded}) and (\ref{eqn:s_0heis}),
\begin{equation}
\label{eqn:P_0}
P_j(m_x,u;\boldsymbol{\xi},\mathbf{v}_j,\mathbf{w}_j)=\left(\frac{u}{\pi}\right)^d.
\end{equation}
The expansion may be integrated term by term.

Recall that we have reduced to the case where 
$\Phi (m_x)\in \mathbb{R}_+\,\boldsymbol{\nu}$, hence
$\Phi(m_x)=\varsigma (m_x)\,\boldsymbol{\nu}$ (see (\ref{eqn:hmzetam})).

\begin{lem}
\label{lem:localization in lambda}
Let $\mathcal{O}'\Subset \mathcal{O}''\subset\mathcal{O}_{\boldsymbol{\nu}}$
be suitably small neighborhoods of $\boldsymbol{\nu}$. 
Let $\varrho_{\boldsymbol{\nu}}\in \mathcal{C}^\infty_c(\mathcal{O}'')$ be
$\mathcal{C}^\infty$, $\ge 0$, and $\equiv 1$ on $\mathcal{O}'$. Then the asymptotics of
 (\ref{eqn:foi+kir1kvarphi}) are unchanged, it the integrand is multiplied by
 $\varrho_{\boldsymbol{\nu}}(\boldsymbol{\lambda})$.
\end{lem}

\begin{proof}
[Proof of Lemma \ref{lem:localization in lambda}]
Since the adjoint action is unitary,
$\mathcal{O}_{\boldsymbol{\nu}}\cap \mathbb{R}_+\,\boldsymbol{\nu}=\{\boldsymbol{\nu}\}$.
Hence, by 
(\ref{eqn:Upsilon_x}) there exists $a_0>0$ such that
$$
\|\partial_{\boldsymbol{\xi}}\Upsilon_x\|=\|u\,\Phi(m_x)-\boldsymbol{\lambda}\|^\varphi\ge a_0,$$ 
for all $u>0$ and $\boldsymbol{\lambda}\in 
\mathrm{supp}(1-\varrho_{\boldsymbol{\nu}})$.
The claim follows integrating by parts in $\boldsymbol{\xi}$, which is legitimate
in view of the cut-off and the exponential factor.

\end{proof}

In the following, we shall redefine 
$\mathcal{C}_{x,k}$ implicitly incorporating the factor $\varrho_{\boldsymbol{\nu}}(\boldsymbol{\lambda})$,
so that integration in $\mathrm{d}V_{\mathcal{O}_{\boldsymbol{\nu}}}(\boldsymbol{\lambda})$
is over $\mathcal{O}''\subset\mathcal{O}_{\boldsymbol{\nu}}$.

We have an equivariant diffeomorphism 
\begin{equation}
\label{eqn:beta orbit}
\beta:g\,T\in G/T\mapsto 
g\cdot\boldsymbol{\nu}
:=\mathrm{Coad}_g(\boldsymbol{\nu})\in \mathcal{O}_{\boldsymbol{\nu}}.
\end{equation}
Let $\mathrm{d}^HV_{G/T}$ be the Haar measure on $G/T$; then
\begin{equation}
\label{eqn:beta covrb}
\beta^*\big(\mathrm{d}V_{\mathcal{O}_{\boldsymbol{\nu}}}  \big)=
\mathrm{vol}(\mathcal{O}_{\boldsymbol{\nu}})\,\mathrm{d}^HV_{G/T}.
\end{equation}
Furthermore, in view of the factor $\beta^*(\varrho_{\boldsymbol{\nu}})$ which is left implicit,
integration over $\mathcal{O}''\subset\mathcal{O}_{\boldsymbol{\nu}}$ in
(\ref{eqn:foi+kir1kvarphi}) gets replaced by integration over a small neighborhood
of $e_G\,T\in G/T$, according 
to (\ref{eqn:beta covrb}). Let us introduce local coordinates on $G/T$ near
$e_G\,T$  by composing the projection $\pi_{G/T}:G\rightarrow G/T$ with the
restriction of the exponential map of $G$ to the Euclidean orthocomplement
$\mathfrak{t}^{\perp_\varphi}\subset \mathfrak{g}$ of $\mathfrak{t}$ w.r.t.
$\varphi$:
$$
E:\boldsymbol{\gamma}\in \mathfrak{t}^{\perp_{\varphi}}
\mapsto e^{\boldsymbol{\gamma}}\,T\in G/T.
$$
When restricted to a small open neighborhood of the origin, $E$ is a diffeomorphism
onto its image, hence a local chart for $G/T$ centered at $e_G\,T$;
we have an isomorphism $\mathfrak{t}^{\perp_{\varphi}}\cong T_{e_G\,T}(G/T)$. The
Lebesgue measure on $\mathfrak{t}^{\perp_{\varphi}}$ associated to the restriction 
$\varphi'$ of $\varphi$ will be denoted $\mathrm{d}^\varphi\boldsymbol{\gamma}$. 
Being $T$-invariant, $\varphi'$ determines an equivariant Riemannian metric 
on $G/T$, whose associated Riemannian density and volume 
will be denoted $\mathrm{d}^\varphi V_{G/T}$ and $\mathrm{vol}^\varphi (G/T)$, respectively.
Then
\begin{equation}
\label{eqn:Estella}
E^*(\mathrm{d}^\varphi V_{G/T})=\mathcal{R}(\boldsymbol{\gamma})\,
\mathrm{d}^\varphi\boldsymbol{\gamma},\quad 
\mathcal{R}(\boldsymbol{\gamma})=1+R_1(\boldsymbol{\gamma}).
\end{equation}
Viewing $G$ as a principal $T$-bundle over $G/T$, by fiber integration one obtains
$\mathrm{vol}^\varphi (G/T)=\mathrm{vol}^\varphi (G)/\mathrm{vol}^\varphi (T)$.
Clearly,
\begin{equation}
\label{eqn:G/Tdensities}
\mathrm{d}^\varphi V_{G/T}=\mathrm{vol}^\varphi (G/T)\,\mathrm{d}^H V_{G/T}.
\end{equation}
Hence, if we view $\beta\circ E$, restricted to a small neighborhood of the
origin in $\mathfrak{t}^{\perp_\varphi}$ as a local coordinate chart on $\mathcal{O}_{\boldsymbol{\nu}}$,
we obtain by (\ref{eqn:Estella}) and (\ref{eqn:G/Tdensities}):
\begin{equation}
\label{eqn:betaEpb}
(\beta\circ E)^*(\mathrm{d}V_{\mathcal{O}_{\boldsymbol{\nu}}})=
\mathrm{vol}(\mathcal{O}_{\boldsymbol{\nu}})\,E^*(\mathrm{d}^HV_{G/H})=
\mathrm{vol}(\mathcal{O}_{\boldsymbol{\nu}})\,
\frac{\mathrm{vol}^\varphi(T)}{\mathrm{vol}^\varphi(G)}\,\mathcal{R}(\boldsymbol{\gamma})\,
\mathrm{d}^\varphi\boldsymbol{\gamma}.
\end{equation}

With these substitutions, recalling 
(\ref{eqn:beta orbit})
we may rewrite (\ref{eqn:foi+kir1kvarphi})
in the following manner:
\begin{eqnarray}
\label{eqn:foi+kir1kvarphiG/T}
\Pi_{k\,\boldsymbol{\nu}}(x_{1,k},x_{2,k})
&\sim&\mathrm{vol}\big(\mathcal{O}_{\boldsymbol{\nu}}\big)^2\cdot
\frac{\mathrm{vol}^{\varphi}(T)}{\mathrm{vol}^{\varphi}(G)^2} \cdot
k^ {1-d_G/2}\,\left(\frac{k}{2\,\pi}\right)^{2\,n_{G}}\,  
\\
&&\int_{1/D}^{D}\,\mathrm{d}u\,
\int_{\mathfrak{g}}\,\mathrm{d}^\varphi\boldsymbol{\xi}\,\int_{\mathfrak{t}^{\perp_\varphi}}\,
\mathrm{d}^\varphi\,\boldsymbol{\gamma}\,\left[
e^{\imath\,\sqrt{k}\,\Upsilon_x (u,\boldsymbol{\xi},e^{\boldsymbol{\lambda}}\cdot \boldsymbol{\nu})}\,
\mathcal{C}_{x,k}\left(u,\boldsymbol{\xi},e^{\boldsymbol{\gamma}}\cdot \boldsymbol{\nu}
;\mathbf{v}_j,\mathbf{w}_j
\right)
\right].\nonumber
\end{eqnarray}
Let us set $g\cdot \boldsymbol{\nu}:=\mathrm{Ad}_g(\boldsymbol{\xi})$, for $g\in G$
and $\boldsymbol{\xi}\in \mathfrak{g}$. We have:
$$
e^{\boldsymbol{\gamma}}\cdot\boldsymbol{\nu}^\varphi=
\boldsymbol{\nu}^\varphi+\left[ \boldsymbol{\gamma}, \boldsymbol{\nu}^\varphi \right]+
R_2(\boldsymbol{\gamma})=
\boldsymbol{\nu}^\varphi-\left[\boldsymbol{\nu}^\varphi, \boldsymbol{\gamma}  \right]+
R_2(\boldsymbol{\gamma}).
$$
Since $\boldsymbol{\lambda}\mapsto \boldsymbol{\lambda}^\varphi$ intertwines the
coadjoint and adjoint actions,
(\ref{eqn:Upsilon_x}) may be rewritten as follows:
\begin{eqnarray}
\label{eqn:Upsilon_xrewritten}
\Upsilon_x (u,\boldsymbol{\xi},e^{\boldsymbol{\gamma}}\cdot \boldsymbol{\nu})&=&
\big\langle u\,\Phi(m_x)-e^{\boldsymbol{\gamma}}\cdot \boldsymbol{\nu},\boldsymbol{\xi}\big\rangle
\nonumber\\
&=&
\varphi \left(  u\,\Phi(m_x)^\varphi-e^{\boldsymbol{\gamma}}\cdot\boldsymbol{\nu}^\varphi,\boldsymbol{\xi}                          \right)\nonumber\\
&=&
\varphi \left(  u\,\Phi(m_x)^\varphi-\boldsymbol{\nu}^\varphi+\left[\boldsymbol{\nu}^\varphi, \boldsymbol{\gamma}  \right]+
R_2(\boldsymbol{\gamma}),\boldsymbol{\xi}\right)\nonumber\\
&=&
\varphi \left(  \big(u\,\varsigma(m_x)-1\big)\,\boldsymbol{\nu}^\varphi+\left[\boldsymbol{\nu}^\varphi, \boldsymbol{\gamma}  \right]+
R_2(\boldsymbol{\gamma}),\boldsymbol{\xi}\right);
\end{eqnarray}
here $R_2(\boldsymbol{\gamma})$ is real-valued.
In terms of the $\varphi$-orthogonal direct sum decompositions
$$
\mathfrak{t}:=\mathrm{span}\left( \boldsymbol{\nu}^\varphi \right)\cap \mathfrak{t}_{\boldsymbol{\nu}},
\quad \mathfrak{g}=\mathrm{span}\left( \boldsymbol{\nu}^\varphi \right)
\oplus \mathfrak{t}_{\boldsymbol{\nu}}\oplus \mathfrak{t}^{\perp_\varphi},
$$
we shall write the general element of $\mathfrak{g}$ as 
$$
\boldsymbol{\xi}=s\,\boldsymbol{\nu}^\varphi_u
+\boldsymbol{\xi}'+\boldsymbol{\xi}'',
\quad \text{where}\quad s\in \mathbb{R},
\,\boldsymbol{\xi}'\in \mathfrak{t}_{\boldsymbol{\nu}},\,
\boldsymbol{\xi}''\in \mathfrak{t}^{\perp_\varphi}.
$$

Furthermore, we may introduce orthnormal basis of $\mathfrak{t}_{\boldsymbol{\nu}}$
and $\mathfrak{t}^{\perp_\varphi}$ w.r.t. $\varphi$, so as to unitarily identify
$\mathfrak{t}_{\boldsymbol{\nu}}\cong \mathbb{R}^{r_G-1}$ and $\mathfrak{t}^{\perp_\varphi}\cong 
\mathbb{R}^{2\,n_{\boldsymbol{\nu}}}$. Let $Z_{\boldsymbol{\nu}^{\varphi}}$ the skew-symmetric and
non-degenerate matrix representing $S_{\boldsymbol{\nu}^\varphi}$ w.r.t. the given 
orthonormal basis of $\mathfrak{t}^{\perp_{\varphi}}$.
Then (\ref{eqn:Upsilon_xrewritten}) may be rewritten:
\begin{eqnarray}
\label{eqn:Upsilon_xrewritten-on}
\lefteqn{
\Upsilon_{x,\boldsymbol{\xi'}} (u,s,\boldsymbol{\xi}'',\boldsymbol{\gamma}):=
\Upsilon_x (u,\boldsymbol{\xi},e^{\boldsymbol{\gamma}}\cdot \boldsymbol{\nu})}\\
&=&
\varphi \left(  \big(u\,\varsigma(m_x)-1\big)\,\boldsymbol{\nu}^\varphi+\left[\boldsymbol{\nu}^\varphi, \boldsymbol{\gamma}  \right]+
R_2(\boldsymbol{\gamma}),s\,\boldsymbol{\nu}^\varphi_u
+\boldsymbol{\xi}'+\boldsymbol{\xi}''\right)\nonumber\\
&=&s\,\big(u\,\varsigma(m_x)-1\big)\,\left\|\boldsymbol{\nu}^\varphi\right\|_\varphi
-\varphi \left( \boldsymbol{\gamma}, \left[\boldsymbol{\nu}^\varphi, \boldsymbol{\xi}''  \right]   \right)
+\varphi \left(
R_2(\boldsymbol{\gamma}),s\,\boldsymbol{\nu}^\varphi_u
+\boldsymbol{\xi}'+\boldsymbol{\xi}''\right)\nonumber\\
&=&s\,\big(u\,\varsigma(m_x)-1\big)\,\left\|\boldsymbol{\nu}^\varphi\right\|_\varphi
-\boldsymbol{\gamma}^t\,Z_{\boldsymbol{\nu}^{\varphi}} \, \boldsymbol{\xi}''   
+ 
R_2(\boldsymbol{\gamma})^t\,\left(s\,\boldsymbol{\nu}^\varphi_u
+\boldsymbol{\xi}'+\boldsymbol{\xi}''\right).\nonumber
\end{eqnarray}

We write (\ref{eqn:foi+kir1kvarphiG/T}) in the form
\begin{eqnarray}
\label{eqn:foi+kir1kvarphiG/Titerated}
\Pi_{k\,\boldsymbol{\nu}}(x_{1,k},x_{2,k})
&\sim&\mathrm{vol}\big(\mathcal{O}_{\boldsymbol{\nu}}\big)^2\cdot
\frac{\mathrm{vol}^{\varphi}(T)}{\mathrm{vol}^{\varphi}(G)^2} \cdot
k^ {1-d_G/2}\,\left(\frac{k}{2\,\pi}\right)^{2\,n_{G}}\,  
\\
&&
\cdot \int_{\mathfrak{t}_{\boldsymbol{\nu}}}\,\mathrm{d}^\varphi\boldsymbol{\xi}'\,
\big[  \mathcal{I}_{x,k}\left( \boldsymbol{\xi}' ;\mathbf{v}_j,\mathbf{w}_j \right) \big],
\nonumber
\end{eqnarray}
where
\begin{eqnarray}
\label{eqn:Ikxi}
\mathcal{I}_{x,k}\left( \boldsymbol{\xi}' ;\mathbf{v}_j,\mathbf{w}_j \right)&:=&
\int_{1/D}^{D}\,\mathrm{d}u\,\int_{-\infty}^{+\infty}\,\mathrm{d}s\,
\int_{\mathfrak{t}^{\perp_\varphi}}\,\mathrm{d}^\varphi\boldsymbol{\xi}''\,
\int_{\mathfrak{t}^{\perp_\varphi}}\,
\mathrm{d}^\varphi\,\boldsymbol{\gamma}\nonumber\\
&&\left[
e^{\imath\,\sqrt{k}\,\Upsilon_{x,\boldsymbol{\xi'}} (u,s,\boldsymbol{\xi}'',\boldsymbol{\gamma})}\,
\mathcal{C}_{x,k}\left(u,\boldsymbol{\xi},e^{\boldsymbol{\gamma}}\cdot \boldsymbol{\nu}
;\mathbf{v}_j,\mathbf{w}_j
\right)
\right].
\end{eqnarray}
We view $\mathcal{I}_{x,k}\left( \boldsymbol{\xi}'  \right)$ as an oscillatory 
integral depending on the parameter $\boldsymbol{\xi}'$, with real phase
$\Upsilon_{x,\boldsymbol{\xi'}} $. Using that $Z_{\boldsymbol{\nu}^{\varphi}} $ is non-degenerate,
and that $\boldsymbol{\gamma}$ is small in norm,
one obtains the following.

\begin{lem}
\label{lem:critical point Upsilon}
For any $\boldsymbol{\xi}'\in \mathfrak{t}_{\boldsymbol{\nu}}$,
$\Upsilon_{x,\boldsymbol{\xi'}} $ has a unique critical point, given by
$$
P_0=(u_0,s_0,\boldsymbol{\xi}''_0,\boldsymbol{\gamma}_0)=
\left( \frac{1}{\varsigma(m_x)}, 0,\mathbf{0},\mathbf{0}     \right).
$$
Hence $\Upsilon_{x,\boldsymbol{\xi'}}(P_0)=0$.
The Hessian matrix at the critical point is 
$$
H_{P_0}(\Upsilon_{x,\boldsymbol{\xi'}})=
\begin{pmatrix}
0&\varsigma(m_x)\,\left\|\boldsymbol{\nu}^\varphi\right\|_\varphi&\mathbf{0}^t&\mathbf{0}^t\\
\varsigma(m_x)\,\left\|\boldsymbol{\nu}^\varphi\right\|_\varphi&0&\mathbf{0}^t&\mathbf{0}^t\\
\mathbf{0}&\mathbf{0}&[0]&Z_{\boldsymbol{\nu}^{\varphi}}\\
\mathbf{0}&\mathbf{0}&-Z_{\boldsymbol{\nu}^{\varphi}}&
\left.\partial^2_{\boldsymbol{\gamma},\boldsymbol{\gamma}}\Upsilon_{x,\boldsymbol{\xi'}}
\right|_{P_0}
\end{pmatrix},
$$
where $[0]$ denotes the zero matrix of order
$(2\,n_G)\times (2\,n_G)$. Hence, its determinant and signature are
$$
\det\big( H_{P_0}(\Upsilon_{x,\boldsymbol{\xi'}})  \big)
=-\varsigma(m_x)^2\,\left\|\boldsymbol{\nu}^\varphi\right\|_\varphi^2\,
\det(Z_{\boldsymbol{\nu}^{\varphi}})^2,\quad
\mathrm{sign}\big( H_{P_0}(\Upsilon_{x,\boldsymbol{\xi'}})   \big)=0.
$$
In particular, $P_0$ is a non-degenerate critical point.
\end{lem}
Integrating by parts in $\mathrm{d}\boldsymbol{\gamma}$ shows that the asymptotics of 
(\ref{eqn:Ikxi}) are unchanged, if the integrand is multiplied by a cut-off function
in $\boldsymbol{\xi}''$, compactly supported and identically equal to $1$ near the origin.

We can apply the stationary phase Lemma. Recalling 
(\ref{eqn:Cxk})-(\ref{eqn:P_0}),  
we obtain for (\ref{eqn:Ikxi}) 
an asymptotic expansion of the form
\begin{eqnarray}
\label{eqn:Ikxi-expanded}
\mathcal{I}_{x,k}\left( \boldsymbol{\xi}' ;\mathbf{v}_j,\mathbf{w}_j \right)
&\sim&\gamma \left( k^{-\epsilon}\,\boldsymbol{\xi}'  \right)\cdot 
\left( \frac{2\,\pi}{k^{1/2}}  \right)^{1+d_G-r_G}\cdot 
\frac{e^{\frac{1}{\varsigma(m_x)}\,\,\left[ 
 \psi_2(\mathbf{w}_1,\mathbf{w}_2)  -\frac{1}{2}\,\|\mathbf{v}_1-\mathbf{v}_2\|_{m_x}^2 
\right]}}{\varsigma(m_x)\,\left\|\boldsymbol{\nu}^\varphi\right\|_\varphi\,
\det(Z_{\boldsymbol{\nu}^{\varphi}})}\nonumber\\
&&\cdot \frac{k^d}{\varsigma(m_x)^d\,\pi^d}\, e^{\frac{1}{\varsigma(m_x)}\,\left[ 
\imath\,\omega_{m_x}\big( \boldsymbol{\xi}_M'(m),\mathbf{v}_1+\mathbf{v}_2\big)
-\frac{1}{2}\,\|    \boldsymbol{\xi}'_M(m_x)  \|_{m_x}^2       \right]}\nonumber\\
&&\cdot\left[ 1+\sum_{j\ge 1}k^{-j/2}\,P'_j(m_x;\boldsymbol{\xi}',\mathbf{v}_j,\mathbf{w}_j)  \right],
\end{eqnarray}
where again $P_j(m_x;\cdot,\cdot,\cdot)$ is a polynomial of degree $\le 3\,j$ and
parity $j$.
We have replaced 
$\|  \boldsymbol{\xi}'_X(x) \|_x$ by $\|  \boldsymbol{\xi}'_M(m_x)  \|_{m_x}$
in view of the fact that $\langle\Phi_M(m_x),\boldsymbol{\xi}'\rangle =0$
since $\boldsymbol{\xi}'\in \mathfrak{t}^{\perp_\varphi}$,
so that $\boldsymbol{\xi}_X'(x) =\boldsymbol{\xi}'_M(m_x)^\sharp$.

The final expansion is obtained by inserting 
(\ref{eqn:Ikxi-expanded}) in (\ref{eqn:foi+kir1kvarphiG/Titerated}) and
integrating term by term. 
The front cut-off, in view of the Gaussian type exponential, may be omitted
without affecting the asymptotics.
The $j$-th summand in
(\ref{eqn:Ikxi-expanded}), $j\ge 0$, contributes by a factor given by 
the
Gaussian type integral
\begin{equation}
\label{eqn:jth-factor}
k^{-j/2}\,\int_{\mathfrak{t}_{\boldsymbol{\nu}}}\,\mathrm{d}^\varphi\boldsymbol{\xi}'\,
\left[e^{\frac{1}{\varsigma(m_x)}\,\left[ 
\imath\,\omega_{m_x}\big( \boldsymbol{\xi}_M'(m_x),\mathbf{v}_1+\mathbf{v}_2\big)
-\frac{1}{2}\,\big\|    \boldsymbol{\xi}'_M(m_x) \big \|_{m_x}^2       \right]} \,
P'_j(m_x;\boldsymbol{\xi}',\mathbf{v}_j,\mathbf{w}_j)\right],
\end{equation}
where we set $P_0=1$.

We compute the leading order term. Recall that we have fixed an orthonormal basis of
$\mathfrak{t}_m'=\mathfrak{t}_{\boldsymbol{\nu}}\cong \mathbb{R}^{r_G-1}$ (\ref{eqn:reductiontmtm'}).
Let $D^\varphi(m_x)$ be as in Definition \ref{defn:Dphim}, and let
$P^\varphi(m_x)$ denote its positive definite square root. 
Furthermore, by Lemma \ref{lem:normal bundle} there exist unique $\boldsymbol{\upsilon}_j\in 
\mathfrak{t}_{\boldsymbol{\nu}}$ such that
$\mathbf{v}_j=J_{m_x}\big( \boldsymbol{{\upsilon}_j}_M(m_x)  \big)$.
Let $\langle\cdot,\cdot\rangle_{st}$ denote the standard scalar product on
$\mathbb{R}^{r_G-1}$,
then
\begin{eqnarray}
\label{eqn:upsilon-Dphi}
\lefteqn{\varsigma(m_x)^{-1}\,
\omega_{m_x}\big( \boldsymbol{\xi}_M'(m),\mathbf{v}_1+\mathbf{v}_2\big)}\\
&=&
\varsigma(m_x)^{-1}\,\omega_{m_x}\Big( \boldsymbol{\xi}_M'(m_x),
J_{m_x}\big( \boldsymbol{{\upsilon}_1}_M(m_x)  \big)+
J_{m_x}\big( \boldsymbol{{\upsilon}_2}_M(m_x)  \big)\Big)\nonumber\\
&=&\varsigma(m_x)^{-1}\, \rho^M_{m_x}\Big( \boldsymbol{\xi}_M'(m_x),
{\boldsymbol{\upsilon}_1}_M(m_x)  +
 {\boldsymbol{\upsilon}_2}_M(m_x)  \Big)     \nonumber\\
 &=& \varsigma(m_x)^{-1}\, {\boldsymbol{\xi}'}^T\, D^\varphi(m_x)\,(  \boldsymbol{\upsilon}_1+
 \boldsymbol{\upsilon}_2)\nonumber\\
 &=&\left\langle\varsigma(m_x)^{-1/2}\,  P^\varphi(m_x)\,\boldsymbol{\xi}',
\varsigma(m_x)^{-1/2}\, P^\varphi(m_x)\,( \boldsymbol{\upsilon}_1+
 \boldsymbol{\upsilon}_2) \right\rangle_{st}.                      \nonumber
\end{eqnarray}
Similarly,
if $\|\cdot\|$ is the standard Euclidean norm then

\begin{eqnarray}
\label{eqn:normsquarexiM}
\varsigma(m_x)^{-1}\,\left\|    \boldsymbol{\xi}'_M(m_x)  \right\|_{m_x}^2    =
\left\|    \varsigma(m_x)^{-1/2}\,P^\varphi(m_x)\,\boldsymbol{\xi} '\right\|^2.
\end{eqnarray}
Hence, setting 
$\boldsymbol{\eta}=\varsigma(m_x)^{-1/2}\,P^\varphi(m_x)\,\boldsymbol{\xi}' $,
we obtain
\begin{eqnarray}
\label{eqn:change of variable gaussian}
\lefteqn{  \int_{\mathfrak{t}_{\boldsymbol{\nu}}}\,\mathrm{d}^\varphi\boldsymbol{\xi}'\,
\left[e^{\frac{1}{\varsigma(m_x)}\,\left[ 
\imath\,\omega_{m_x}\big( \boldsymbol{\xi}_M'(m_x),\mathbf{v}_1+\mathbf{v}_2\big)
-\frac{1}{2}\,\big\|    \boldsymbol{\xi}'_M(m_x)  \big\|_{m_x}^2       \right]} \,
\right]  }\nonumber\\
&=&\frac{\varsigma(m_x)^{\frac{r_G-1}{2}}}{\det\left(P^\varphi(m_x)\right)}\,
\int_{\mathbb{R}^{r_G-1}}\,\mathrm{d}\boldsymbol{\eta}\,
\left[e^{ 
\imath\, \left\langle\boldsymbol{\eta},
\varsigma(m_x)^{-1/2}\, P^\varphi(m_x)\,( \boldsymbol{\upsilon}_1+
 \boldsymbol{\upsilon}_2) \right\rangle_{st}
-\frac{1}{2}\,\|    \boldsymbol{\eta} \|^2       } \,
\right] \nonumber\\
&=&\frac{(2\,\pi)^{\frac{r_G-1}{2}}\,\varsigma(m_x)^{\frac{r_G-1}{2}}}{\mathcal{D}^\varphi (m)}\,
e^{-\frac{1}{2\,\varsigma(m_x)}\,\|\mathbf{v}_1+\mathbf{v}_2\|_m^2}.
\end{eqnarray}
Plugging (\ref{eqn:change of variable gaussian}) into (\ref{eqn:Ikxi-expanded}) and then
in (\ref{eqn:foi+kir1kvarphiG/Titerated}) we obtain the leading order term
in the statement of the Theorem. The other terms can be handled similarly.

\end{proof}

\end{document}